\documentclass[11pt, reqno]{amsart}
\usepackage{graphicx}
\usepackage{amsthm, amsfonts, amssymb, amsmath, latexsym, array,color,amscd}
\usepackage{mathtools}
\usepackage{mathrsfs}
\usepackage{bm}
\usepackage{kbordermatrix}
\usepackage{enumerate}
\usepackage[margin=3cm]{geometry}
\usepackage{hyperref}
\usepackage{verbatim}
\hypersetup{
	colorlinks=true, 
	linkcolor=blue, 
	citecolor=blue, 
	filecolor=blue,
	urlcolor=blue
}
\usepackage[normalem]{ulem}
\usepackage{tikz-cd}
\usepackage{todonotes}
\usepackage{tikz}
\usepackage[capitalise, noabbrev]{cleveref}
\graphicspath{{Pictures/}}

\title{Trivariate Splines on Fans of Hyperplane Arrangements and Koszul Homology}

\author[]{Carles Checa}
\address[C. Checa]
{Department of Mathematics, University of Copenhagen, Universitetsparken 5, 
2100, Copenhagen, Denmark}
\email{ccn@math.ku.dk}

\author[]{Michael DiPasquale}
\address[M. DiPasquale]
{Department of Mathematical Sciences, New Mexico State University, 1290 Frenger Mall
MSC 3MB / Science Hall 236, 
Las Cruces, New Mexico 88003-8001, USA}
\email{midipasq@nmsu.edu}

\author[]{Pablo Mazón}
\address[P. Mazón]
{Department of Mathematics, CUNEF Universidad, Madrid, Spain}
\email{pablo.mazon@cunef.edu}

\author[]{Th\'ai Th\`anh Nguy$\tilde{\text{\^E}}$n}
\address[T. T. Nguy$\tilde{\text{\^e}}$n]{Department of Mathematics, University of Dayton, 300 College Park, Dayton, OH 45469, USA  and \\
	University of Education, Hue University, 34 Le Loi St., 
    Hue, Viet Nam}
\email{tnguyen5@udayton.edu}

\author[]{Liana Sega}
\address[L. Sega]
{Division of Computing, Analytics and Mathematics, University of Missouri-Kansas City, USA}
\email{segal@umkc.edu}

\author[]{Prajwal Udanshive}
\address[P. Udanshive]
{Department of Mathematics, Western University, 1151 Richmond Street, London, Ontario,  N6A 3K7, Canada}
\email{pudanshi@uwo.ca}

\author[]{Adam Van Tuyl}
\address[A. Van Tuyl]
{Department of Mathematics and Statistics
McMaster University, Hamilton, ON L8S 4L8, Canada}
\email{vantuyla@mcmaster.ca}

\author[]{Nelly Villamizar}
\address[N. Villamizar]
{Department of Mathematics, Swansea University, Fabian Way, Swansea, SA1 8EN, United Kingdom}
\email{n.y.villamizar@swansea.ac.uk}


\date{\today}

\newcommand{\RR}{\mathbb{R}}
\newcommand{\R}{\mathbb{R}}
\newcommand{\ZZ}{\mathbb{Z}}
\newcommand{\PP}{\mathbb{P}}
\newcommand{\Z}{\mathbb{Z}}
\newcommand{\calJ}{\mathcal{J}}
\newcommand{\calS}{\mathcal{S}}
\newcommand{\calR}{\mathcal{R}}
\newcommand{\calL}{\mathscr{L}}

\newcommand{\A}{\mathscr{A}}
\newcommand{\by}{{\boldsymbol{y}}}
\newcommand{\K}{{\sf{K}}}
\newcommand{\kk}{{\Bbbk}}
\newcommand{\HH}{{\mathrm{H}}}

\newcommand{\HK}{{\sf{H}}}
\newcommand{\HP}{{\sf{HP}}}

\newcommand{\br}{{\bm{r}}}

\DeclarePairedDelimiter{\ceil}{\lceil}{\rceil}

\newcommand{\syz}{\mathrm{Syz}}
\newcommand{\reg}{\mathrm{reg}}
\newcommand{\grade}{\mathrm{grade}}
\newcommand{\Hom}{\mathrm{Hom}}
\newcommand{\Ext}{\mathrm{Ext}}

\newcommand{\hilb}{\mathrm{Hilb}}

\newcommand{\cv}{\ensuremath{\bm{0}}}

\makeatletter
\newcommand*\bigcdot{\mathpalette\bigcdot@{.7}}
\newcommand*\bigcdot@[2]{\mathbin{\vcenter{\hbox{\scalebox{#2}{$\m@th#1\bullet$}}}}}
\makeatother

\usepackage[shortlabels]{enumitem}
\newtheorem{theorem}{Theorem}[section]
\newtheorem{lemma}[theorem]{Lemma}
\newtheorem{proposition}[theorem]{Proposition}
\newtheorem{corollary}[theorem]{Corollary}

\theoremstyle{definition}
\newtheorem{definition}[theorem]{Definition}
\newtheorem{notation}[theorem]{Notation}
\newtheorem{example}[theorem]{Example}

\theoremstyle{remark}
\newtheorem{remark}[theorem]{Remark}
\numberwithin{equation}{section}

\begin{document}

\begin{abstract}
We study the space of splines $\calS^{\bm r}(\Sigma^\A)$ 
where ${\bm r}$ denotes a smoothness distribution and $\Sigma^\A$ is the fan of a central hyperplane arrangement $\A$ in $\R^3$. This is the first step in the analysis of splines on three-dimensional cross-cut partitions, which naturally generalize planar cross-cut partitions.
We show that the Hilbert function of $\calS^{\bm r}(\Sigma^\A)$ is bounded by an expression that involves the dimensions of specific Koszul homology modules constructed
from the defining equations of the hyperplane arrangement $\A$
and the smoothness distribution function. 
By
exploiting this connection
with Koszul homology, we are able
to: 1) compute the dimension of the spline space in high degrees,
2) compute all values of the dimension 
of the spline space if $\A$ is \textit{generic} with five or fewer hyperplanes, and 3) compute the Hilbert function of the spline
space if $\A$ is a generic arrangement with sufficiently many hyperplanes and ${\bm r}$ is a constant distribution.  As an application of our methods, we compute $\dim \calS^0_d(\Sigma^\A)$ and $\dim \calS^1_d(\Sigma^\A)$ for all values of $d$ when $\A$ is a generic arrangement.
\end{abstract}
\maketitle

\section{Introduction}

For our purposes, a multivariate spline is a piecewise polynomial function on a subdivision of some domain in 
$\R^m$ with prescribed smoothness across faces of the subdivision.  Splines play a central 
role in geometric modeling, finite element methods, and approximation 
theory  (see, e.g., \cite{LaiSchumaker} for general background).

The study of splines has long been shaped by
connections with commutative and homological algebra, beginning with
the work of Chui and Wang \cite{CW83}, Billera \cite{Homology}, and
Billera and Rose \cite{DimSeries}, which introduced algebraic and
homological techniques for understanding the dimension and structure
of spline spaces.
Schenck and Stillman refined these methods, making a crucial connection to ideals generated by powers of linear forms
\cite{Spect,MinReg,LCoho,FatPoints}.
Within the last decade, these ideas have been extended to include splines with enhanced smoothness across lower-dimensional
faces \cite{Deepesh-Nelly-supers} and explicit lower bounds on the
dimension of splines on three-dimensional complexes
\cite{Michael-Nelly-SIAGA,Michael-Nelly-CT}.
Additional developments include geometric
conditions on two-dimensional partitions under which the dimension of the spline space
admit a combinatorial formula \cite{MichaelBeihui, T-splines} and an important counterexample to the Schenck-Stiller `$2r+1$' conjecture, showing that the dimension of spline spaces on triangulations can have unexpected dimension in relatively large degree
\cite{Beihui-Hal-Mike,YS19}.
These approaches have further been generalized to semialgebraic splines \cite{Stiller,Semialgebraic,BivariateSemialgebraic}, to algebraic frameworks for geometrically
continuous splines \cite{geometricContinuity}, and continue to inspire
ongoing work, much of which is discussed in the surveys~\cite{S16} and~\cite{Martina-Hal-Julina}.

In this paper, we consider splines on a subdivision of $\R^3$ induced by a collection of planes through the origin.  More precisely, let $\A=\{H_1,\ldots,H_n\}$ where each $H_i$ is a plane that contains the origin; $\A$ is called a \textit{central hyperplane arrangement}, where the adjective \textit{central} indicates that each hyperplane contains the origin.  The hyperplane arrangement $\A$ subdivides $\R^3$ into a collection of regions.  For instance, one plane divides $\R^3$ into two regions, two distinct planes divide $\R^3$ into four regions, and three distinct planes may divide $\R^3$ into six regions (if the three planes contain a common line) or eight regions (if the three planes do not contain a common line).  For an arbitrary hyperplane arrangement $\A\subset\R^3$, the data of the subdivision induced by $\A$ is encoded by the \textit{fan} $\Sigma^\A$ of $\A$ (see \Cref{sec:Preliminaries}).  We consider a smoothness distribution $\br:\A\to \Z_{\ge 0}$ and we study the space of functions $\calS^\br(\Sigma^\A)$ which are piecewise polynomial on the subdivision of $\R^3$ induced by $\A$ and continuously differentiable to order $\br(H_i)$ across each plane $H_i$.  We say that a function in $\calS^\br(\Sigma^\A)$ is a \textit{spline on the fan} $\Sigma^\A$.  For a given non-negative integer $d$, our primary objective is to compute the dimension of the space $\calS^\br_d(\Sigma^\A)$ of functions $F\in\calS^\br(\Sigma^\A)$ where the restriction of $F$ to each region is a homogeneous polynomial of degree $d$.

Our setting is motivated by planar \emph{cross-cut partitions}; these are partitions of a planar region by straight lines joining boundary points~\cite{CW83,LCoho}.  
The natural generalization to $\RR^3$ is a partition of a region by a collection of planes; let us call this a \textit{three-dimensional cross-cut partition}.  Our aim in this paper is to study the local case of a three-dimensional cross-cut partition.  That is, if we consider a point at which some number of planes meet in a three-dimensional cross-cut partition and the cells of the partition surrounding that point, we obtain (up to translation) a partition by planes that pass through the origin.  A full study of the local case is a natural first step in the analysis of three-dimensional cross-cut partitions, which will be a topic of future research.  This local-to-global analysis for splines is well established; for instance, Schumaker's formula for splines on two-dimensional cells yields the fundamental non-trivial expression in his lower bound~\cite{SchumakerLower}.  See also the two sets of companion papers~\cite{Semialgebraic,BivariateSemialgebraic} and~\cite{Michael-Nelly-SIAGA,Michael-Nelly-CT} which both follow the local-to-global strategy.

Planar cross-cut partitions are some of the first partitions for which dimension formulas were computed in all degrees.  For constant smoothness distribution $\bm r$, the dimension of the spline space over a cross-cut partition was computed by Chui and Wang~\cite[Lemma 4.1]{CW83} when no three lines of the cross-cut partition meet at a point.  For arbitrary cross-cut partitions, Schenck and Stillman~\cite[Theorem 5.3]{LCoho} proved that the module of splines is a free module over the polynomial ring and thus the dimension of splines over a cross-cut partition is always given by a well-known lower bound of Schumaker~\cite{SchumakerLower}.

The situation for splines on the fan $\Sigma^\A$ of a hyperplane arrangement $\A\subset\RR^3$ is significantly more complicated. 
The module $\calS^\br(\Sigma^\A)$ is not usually free and even if some genericity conditions are imposed on $\A$, we show in \Cref{thm:C1} that $\dim\calS^1_4(\Sigma^\A)$ can exhibit a dependence on geometry that is strikingly similar to the well-known Morgan-Scott split (c.f. Section~6 of \cite{AlfeldSurvey}).

The novelty in this paper is to exhibit
a connection between the Hilbert functions of these 
spline spaces and the first Koszul homology modules constructed
from the linear forms defining $\A$. 
To state one of our main results,
if $\A$ is the collection of hyperplanes 
$\A = \{H_1,\dots,H_n\}$, let $\ell_i$ be a linear form in $R = \mathbb{R}[x,y,z]$ 
that vanishes on $H_i$, for $i=1,\ldots,n$. 
Suppose $\br:\A\to\ZZ_{\ge 0}$ is a smoothness distribution on $\A$.
We let 
$\bm{\ell}^\br =(\ell_1^{\br(H_1)+1},\ldots,
\ell_n^{\br(H_n)+1})$
denote the sequence of powers of these linear forms. 
The following is the main result of the paper in simplified form when $\A$ is \textit{generic} (that is, no three hyperplanes intersect along a line).

\begin{theorem}[Theorem \ref{thm.SplineGenericArrangement}]
\label{maintheorem}
Suppose $\A$ is a generic hyperplane arrangement in $\RR^3$ of $n$ hyperplanes and $\br\colon\A\to\ZZ_{\ge -1}$ is a smoothness distribution. If $\Sigma^\A$ is the fan of $\A$, 
then for all $d\ge 0$ the dimension of the spline space $\calS^{\br}_d(\Sigma^\A)$ is given by 
\[
\dim\calS^{\br}_d(\Sigma^\A)=2\binom{d+2}{2}+2\sum_{\substack{H_1,H_2\in \A\\ H_1\neq H_2}}\binom{d-\br(H_1)-\br(H_2)}{2}+\dim \HK_1(\bm{\ell}^{\br})_d-\dim \HK_0(\bm{\ell}^{\br})_d
,
\]
where $\HK_i(\bm{\ell}^{\br})$ denotes the $i$-th Koszul homology module
of $\bm{\ell}^{\br}$(see \Cref{sec:KoszulConnection}).
\end{theorem}

\noindent
To prove \Cref{maintheorem}, we use the \textit{Billera-Schenck-Stillman chain complex}, which is a standard tool to study splines from an algebraic and homological perspective (see \Cref{sec.homologicalmethodsforsplines}).  The main point is that $\calS_d^\br(\Sigma)$ is (essentially) the top homology of a chain complex $\calJ[\Sigma]_\bullet$.  Thus, using the Euler-Poincar\'e characteristic, one can compute $\dim \calS^\br_d(\Sigma)$ provided one knows the Hilbert functions of the modules in the chain complex $\calJ[\Sigma]_\bullet$ and the Hilbert functions of the remaining homology modules of $\calJ[\Sigma]_\bullet$ (see \Cref{prop:EulerCharacteristicAndDimension}).  Our main insight is that, in the case of the fan $\Sigma^\A$ of a hyperplane arrangement $\A$, the module $H_1(\calJ[\Sigma^\A]_\bullet)$
can be related to the Koszul homology of $\bm{\ell}^{\br}$. 
In particular, we show in \Cref{thm.linkhomologoies} that $\dim \calS^\br_d(\Sigma^\A)$ can be bounded by an expression that involves the dimension of Koszul homology modules.  This bound is an equality if $\A$ is generic, giving the statement of \Cref{maintheorem}.

By unlocking this connection between $\dim\calS^{\br}_d(\Sigma^\A)$
and $\HK_i(\bm{\ell}^{\br})$, we can derive new results
about the Hilbert function of $\calS^{\br}(\Sigma^\A) = \bigoplus_{d \geq 0} \calS^\br_d(\Sigma^\A)$
by employing results and techniques related to Koszul homology.
We highlight some of these new results.  

In one direction, if $d \gg 0$, then the
dimensions of the Koszul homology modules in \Cref{maintheorem}
become zero, so the formula for $\dim\calS^{\br}_d(\Sigma^\A)$
reduces to the first two terms in the statement.  In fact, this is true for arbitrary arrangements (not just generic arrangements), as we show in \Cref{thm.SplineGenericArrangement}.  We next take up how large $d$ must be for this vanishing to occur.  This is given by the Castelnuovo-Mumford regularity of the Koszul homology modules.  The regularity of the Koszul homology modules can be bounded by the regularity of the ideal generated by the elements of the sequence $\bm{\ell}^\br$, which in turn can be bounded by the values of the smoothness distribution $\br$ and an interpolation statistic of the hyperplane arrangement (see \Cref{maintheoremUsingRegularity}).

In the remainder of the paper (\Cref{sec:smallnumber} to \Cref{sec:C0C1}) our objective is to compute $\dim\calS^\br_d(\Sigma^\A)$ for \textit{all} integers $d\ge 0$ for certain \textit{generic} arrangements.  For instance, if $\A$ is generic and has three, four, or five  hyperplanes, we compute 
$\dim\calS^{\br}_d(\Sigma^\A)$ for all $d$.  For
the case of three or four hyperplanes (\Cref{thm:threefourHyperplanes}), we can explicitly
compute the dimensions of the Koszul homology modules.
In the case of five hyperplanes given in \Cref{thm:fiveHyperplanes},
we show that computing the desired Hilbert function reduces
to computing the Hilbert function of $R/\langle \bm{\ell}^{\br} \rangle$, where $\langle\bm{\ell}^{\br}\rangle$ is the ideal generated by the sequence $\bm{\ell}^{\br}$.
For five hyperplanes, we describe an algorithm that uses inverse 
systems to reduce the computation to computing the Hilbert function
of a set of fat points on a conic in $\mathbb{P}^2$.  One 
can then use the algorithm of Catalisano \cite{Catalisano-1991}
to compute this new Hilbert function.  

If $\A$ is generic and has 
three, four, or five hyperplanes, then
the Hilbert function of $\calS^{\br}(\Sigma^\A)$ depends only upon
the number of hyperplanes and the smoothness distribution.  However,
as we show in \Cref{ex:sixhyperplanes} and \Cref{thm:C1}, if $\A$ is generic with six or more hyperplanes,
then one must take into consideration the equations of the hyperplanes.  This is a consequence of the fact that, even if $\A$ is generic, $\HK_1(\bm{\ell}^{\br})$ 
depends upon the choice of the linear forms $\ell_1,\ldots,\ell_n$
if $n \geq 6$.

We next restrict to the case when $\A$ has a constant smoothness distribution (that is, $\br(H_i) = r$ for all $1\le i\le n$ for a fixed integer $r \geq 0$) and many hyperplanes.  More precisely, if $\A$ is sufficiently generic and has at least $\binom{r+3}{2}$ hyperplanes, in \Cref{theorem: spline dimension many planes} we compute $\dim\calS^\br_d(\Sigma^\A)$ for all $d\ge 0$.  In \Cref{sec:C0C1}, we close out the paper by applying our results to compute the dimension of $C^0$ and $C^1$ splines on fans of generic arrangements in all degrees.

\subsection{Structure of the paper}
Our paper is structured as follows. In \Cref{sec:Preliminaries}, we provide
the needed background and terminology related to splines,
hyperplane arrangements, and other relevant topics.   \Cref{sec.homologicalmethodsforsplines} introduces the needed homological machinery, including
\Cref{prop:EulerCharacteristicAndDimension} which relates
the dimensions of the graded spline spaces to the homology
of certain chain complexes.  We then provide some useful
presentations for these homology groups.  In \Cref{sec:KoszulConnection}, we 
review the needed background on Koszul homology.  Then,
in the case our fan is the fan of a hyperplane
arrangement, we relate
the homology groups of the previous section with the Koszul
homology modules constructed from the defining equations of the 
hyperplanes.  This leads to our first main result 
(\Cref{thm.SplineGenericArrangement}). 
\Cref{sec.regularity,sec:smallnumber,sec.uniformsmoothness,sec:C0C1}
then focus on deriving consequences from this
connection, including the results described above.

\section{Preliminaries}
\label{sec:Preliminaries}

In this section, we recall the relevant definitions and results from the theory of cones, fans, hyperplane arrangements, and graded modules.

Throughout, let $R=\RR[x_1,\ldots,x_m]$ be the $\RR$-algebra of polynomials in $m$ variables, graded by total degree, and let $R_d$ denote the $\RR$-vector space of homogeneous polynomials of degree $d$.
A \emph{linear form} means an element of $R_1$ (hence homogeneous, with no constant term).

\subsection{Cones, fans, and hyperplane arrangements}
We follow~\cite{Zie} for cones and fans, and we use~\cite{OrlikTerao} as a standing reference for hyperplane arrangements (see also \cite{Stanley2007} for a standard introductory survey on hyperplane arrangements).

\subsubsection{Hyperplanes, half-spaces, and cones}
A \textit{hyperplane} $H=H_{\ell}$ in $\RR^m$ is the zero locus of a linear form $\ell\in R_1$, i.e.
$
H_\ell=\{\bm{x}\in\RR^m : \ell(\bm{x})=0\}.
$
The associated closed half-spaces are
\begin{equation}\label{eq:halfSpaces}
H_{\ell,+}=\{\bm{x}\in\RR^m : \ell(\bm{x})\ge 0\}, \qquad
H_{\ell,-}=\{\bm{x}\in\RR^m : \ell(\bm{x})\le 0\},
\end{equation}
and we also write $H_{\ell,0}=H_\ell$.  Observe that $H_{\ell,+}\cap H_{\ell,-}=H_{\ell,0}$.

A 
\textit{cone} in $\RR^m$ is a finite intersection of half-spaces of the form~\eqref{eq:halfSpaces}; its \textit{dimension} is the dimension of its linear span.
A \textit{face} of a cone is the intersection of the cone with a supporting hyperplane (hence again a cone).
A cone is called \textit{essential} if $\{\cv\}$ is a face of the cone; equivalently, it contains no positive-dimensional linear subspace.

\begin{definition}
\label{def:fan}
A \textit{fan} $\Sigma\subseteq\RR^m$ is a collection of cones $\{\sigma_1,\ldots,\sigma_s\}$ with the two properties:
\begin{enumerate}
\item every face of a cone in $\Sigma$ is also in $\Sigma$, and
\item the intersection of any two cones in $\Sigma$ is a face of both.
\end{enumerate}
We call a fan $\Sigma$ \textit{complete} if $\bigcup_{\sigma\in\Sigma}\sigma=\RR^m$.
\end{definition}

For a fan $\Sigma\subseteq\RR^m$, we put $|\Sigma|=\bigcup_{\sigma\in\Sigma}\sigma$. 
The \textit{dimension} of $\Sigma$ is the maximal dimension of a cone in $\Sigma$.
Maximal cones under inclusion are called \textit{facets}, and $\Sigma$ is \textit{pure} if all facets have the same dimension.
We call $\Sigma$ \textit{essential} if $\{\cv\}$ is a cone of $\Sigma$ (equivalently, every cone of $\Sigma$ is essential).

\begin{remark}
Ziegler~\cite{Zie} calls a cone (resp$.$ fan) \textit{pointed} if $\{\cv\}$ is a face of the cone (resp. a cone of the fan).
We use the adjective \textit{essential} to align terminology with that of hyperplane arrangements.
\end{remark}

If $\Sigma\subseteq\RR^m$ is a pure $m$-dimensional fan, a cone $\sigma\in\Sigma$ is an \textit{interior face} if it is contained in the intersection of two facets of $\Sigma$, and a \textit{boundary face} otherwise.
For $0\le i\le m$, let $\Sigma_i$ be the set of $i$-dimensional cones of $\Sigma$, and let $\Sigma_i^\circ$ be the subset of interior $i$-dimensional cones.

A pure $m$-dimensional fan $\Sigma\subseteq\RR^m$ is \textit{hereditary} if, for any two facets $\sigma$ and $\sigma'$, $\sigma$ can be connected to $\sigma'$ by a sequence of facets
$\sigma=\sigma_1,\sigma_2,\ldots,\sigma_k=\sigma'$ so that $\sigma\cap\sigma'\subseteq \sigma_i$ for all $1\le i\le k$ and $\dim(\sigma_i\cap\sigma_{i+1})=m-1$ for $1\le i\le k-1$. 

\begin{remark}
A complete fan $\Sigma\subset \RR^m$ is always hereditary. Henceforth, when we refer to a fan $\Sigma\subseteq\RR^m$, it will always be pure and $m$-dimensional.
\end{remark}

\subsubsection{Arrangements}
A \emph{hyperplane arrangement} in $\RR^m$
is a finite collection of hyperplanes  $\A=\{H_1,\dots,H_n\}$.  By definition, each hyperplane is a linear space and so contains the origin in $\R^m$.

The \emph{intersection lattice} of $\A$ is $\mathcal L(\mathcal A)=\{\cap_{H\in F} H : F\subseteq \mathcal A\}$
that is, the collection of all subspaces obtained by intersecting some subset 
of hyperplanes in $\A$ (including $\R^m$ itself, as the empty intersection), 
partially ordered by reverse inclusion.  We write $\calL_i(\A)=\{X\in\calL(\A):\mathrm{codim}(X)=i\}$.

We say that $\A$ is \emph{essential} if $\mathrm{codim}(\bigcap_{H\in\A}H)=m$, equivalently, $\bigcap_{H\in\mathcal A} H=\{\mathbf \cv\}.$
An arrangement $\A$ in $\R^m$ is \emph{generic} if for every choice of
$k\le m$ distinct hyperplanes $H_{i_1},\dots,H_{i_k}$ in $\A$ gives a rank $k$ flat, i.e., we have
$
\mathrm{codim}(H_{i_1}\cap\cdots\cap H_{i_k})=k.
$
Observe that a generic arrangement in $\RR^m$ with at least $m$ hyperplanes is necessarily essential.
\subsubsection{The fan of an arrangement}
For a central arrangement $\A=\{H_{\ell_1},\ldots,H_{\ell_n}\}$ in $\RR^m$ and $\epsilon=(\epsilon_1,\ldots,\epsilon_n)\in\{-,0,+\}^n$, define
$C_\epsilon=\bigcap_{i=1}^n H_{\ell_i,\epsilon_i}$ 
and $\Sigma^\A=\{\,C_\epsilon : \epsilon\in\{-,0,+\}^n\},$
where $H_{\ell_i,\epsilon_i}$ are the spaces defined in \cref{eq:halfSpaces}.
One can check that $\Sigma^\A$ is a fan, which we call the \textit{fan of} $\A$.  It is a complete, pure, $m$-dimensional fan.  
Moreover, $\Sigma^\A$ is essential if and only if $\A$ is essential.

The facets of $\Sigma^\A$ are the closures of the connected components of
$\RR^m\setminus \bigcup_{H\in\A}H$, i.e.\ the \emph{regions} (often called \emph{chambers})
of $\A$.

\subsection{Graded modules and Hilbert function}\label{ss:HilbertPoly} 
Following the notation above, the polynomial ring $R=\RR[x_1,\ldots,x_m]$ carries its standard grading $R=\bigoplus_{d\ge0} R_d,$
where multiplication maps $R_i\times R_j$ into $R_{i+j}$.
In a similar way, an $R$-module $M$ is called \textit{graded} if there is a direct sum decomposition $M\cong \bigoplus_{i\in \ZZ} M_i$, where $M_i$ is an $\RR$-vector space for all $i \in \ZZ$ and there is a multiplication map $R_i\times M_j\to M_{i+j}$ given by $(f,m)\to fm$. 
Our main reference of this section is \cite[Chapter 1]{CommAlg}, see also \cite[Chapter 1]{cohenMacaulay}.

\begin{definition}
If $M= \bigoplus_{i\in\Z}M_i$ is a graded $R$-module, the \textit{Hilbert function} of $M$ 
 is given by the dimensions of its graded pieces, $i\mapsto \dim M_i$.
The \textit{Hilbert series} of $M$, denoted $\hilb_M(t)$, is the generating function of the Hilbert function, namely
$
\hilb_M(t)=\sum_{i\in \Z} \dim(M_i)t^i
$.

For $a\in \ZZ$, we define $M(a)$ as the module isomorphic to $M$ whose grading is defined by $M(a)_{i}=M_{a+i}$. 
We say that $M(a)$ is the module $M$ with grading \emph{shifted} by $a$.
\end{definition}

 If $M$ is a finitely generated graded $R$-module, then the  Hilbert function $\dim M_i$ of $M$ is an alternating sum of binomial coefficients of the form $\binom{i+m-1-a}{m-1}$, for some $a\in\Z$, where we adopt the convention that $\binom{A}{B} = 0$  whenever $A<0$ or $B>a$.  If $i\gg 0$ then the values of $\dim M_i$ agree with a polynomial in the variable $i$, called the \textit{Hilbert polynomial} of $M$ and denoted $\HP_M(i)$.

Moreover,  if $M$ is a finitely generated graded $R$-module, the Hilbert series of $M$ is a rational function and has the form
$
\hilb_M(t)=\frac{p(t)}{(1-t)^{m}},
$
where $p(t)$ is a polynomial with integer coefficients (e.g.,
see \cite[Proposition 4.4.1]{cohenMacaulay}).

\subsection{Conventions for direct sums}\label{ss:indexing_conventions_direct_sums}
Let $\{M_w\}_{w\in W}$ be a finite collection of $R$-modules indexed by a finite set $W$, with $|W|$ elements. An element of
$\bigoplus_{w\in W} M_w$ is written either as $(m_w)_{w\in W}$ or as
$\sum_{w\in W} m_w b_w$, where $b_w$ denotes the standard basis symbol of the
$w$-th summand. We also write $\bigoplus_{w\in W} M_w b_w$.
In the special case where $M_w=M$ for all $w\in W$, we also write $\bigoplus_{w\in W} M b_w$ as either $M^W$ or $M^{|W|}$.
For example, if $\Sigma\subseteq \RR^m$ is a pure $m$-dimensional fan with $f$ facets, we write either $R^{\Sigma_m}$ or $R^f$ for $\bigoplus_{\sigma\in\Sigma_m} R b_\sigma$.
	
Suppose now that each $M_w$ is a graded cyclic $R$-module with homogeneous generator $e_w$ of degree $a_w$. By abuse of notation, we also denote by $e_w$
the corresponding generator of the $w$-th summand in
$\bigoplus_{w\in W} M_w$. Then
\[\bigoplus_{w\in W} M_w \cong \bigoplus_{w\in W} R e_w
\quad\text{and}\quad
\bigoplus_{w\in W} R(-a_w)b_w \cong \bigoplus_{w\in W} R e_w,\]
via the graded $R$-module isomorphism
$\sum_{w\in W} g_w b_w \longmapsto \sum_{w\in W} g_w e_w$.

\subsection{Splines on fans}
Let $\Sigma\subseteq\RR^m$ be a pure $m$-dimensional fan.  For each $\tau \in \Sigma_{m-1}$, let $\ell_\tau \in R$ be a linear form vanishing on $\tau$.  A map $\br\colon \Sigma_{m-1} \rightarrow \mathbb{Z}_{\geq -1}$ is called a \emph{smoothness distribution}.  For a function $F\colon|\Sigma|\to\RR$, we write $F_\sigma=F|_\sigma$ for the restriction of $F$ to $\sigma\in \Sigma_m$.

\begin{definition}\label{def:splines}
Let $\Sigma\subseteq\RR^m$ be a pure $m$-dimensional fan with a smoothness distribution $\br\colon\Sigma_{m-1}\to \mathbb{Z}_{\geq -1}$.  The space of \emph{splines} $S^\br(\Sigma)$ on the pair $(\Sigma,\br)$ is defined as
\begin{align}
\calS^\br(\Sigma) = \bigl\{(F_\sigma)_{\sigma \in \Sigma_m}\in R^{\Sigma_m} \colon &F_{\sigma_1} - F_{\sigma_2} \in \bigl\langle \ell_\tau^{\br(\tau)+1} \bigr\rangle, \;\label{def:splines}\\ 
&\text{for all\;} \sigma_1, \sigma_2 \in \Sigma_m \text{\; such that\; } \sigma_1 \cap \sigma_2 = \tau \in \Sigma_{m-1}\bigr\}.\nonumber
\end{align}
\end{definition}

If the smoothness distribution $\br$ is constant, i.e., $\br(\tau) = r$ for all $\tau \in \Sigma_{m-1}$, then in \cref{def:splines} we simply write $r$ instead of $\br(\tau)$, and $\calS^r(\Sigma)$ instead of $\calS^\br(\Sigma)$.

\begin{remark}
Splines are usually defined as piecewise polynomial functions that are differentiable to some order.  In fact, $\calS^\br(\Sigma)$ consists of those functions $F\colon |\Sigma|\to \RR$ so that $F|_{\sigma}$ is a polynomial for every $\sigma\in \Sigma_m$ and $F$ is differentiable to order $\br(\tau)$ across $\tau$ for every interior face $\tau\in\Sigma^\circ_{m-1}$. 
 If $\Sigma$ is hereditary, it follows from \cite[Proposition~1.2]{DimSeries} that for every face $\gamma\in\Sigma_{m-i}$, for $i\geq 1$, a spline $F\in \calS^\br(\Sigma)$ is differentiable of order $\min\{\br(\tau)\colon\tau\supseteq \gamma\}$ across $\gamma$, and differentiable to order $\min\{\br(\tau):\tau\in\Sigma_{m-1}\}$ at the origin.
\end{remark}

\begin{remark}
In this paper, we consider smoothness distributions that may vary across codimension one faces.  It is also possible to impose an order of smoothness at codimension two or higher faces that exceeds the order of smoothness across codimension one faces.  In this case, the resulting splines are said to have \emph{enhanced} smoothness at those faces, and are referred to as \emph{supersmooth splines} or \emph{supersplines}; see~\cite{Deepesh-Nelly-supers}.
\end{remark}

For $d\geq 0$, we define
\[
\calS^\br_d(\Sigma) = \{F \in \calS^\br(\Sigma) \colon F_\sigma \in R_d \; \text{for all \;} \sigma \in \Sigma_m\}.
\]
Notice that $\calS^\br_d(\Sigma)$ is an $\mathbb{R}$-vector subspace 
of $R^{\Sigma_m}$. 
In fact, 
$\calS^\br(\Sigma) \cong \bigoplus_{d \geq 0} \calS^\br_d(\Sigma)$ is a graded $\RR$-algebra~\cite[Lemma~2.5]{DimSeries}.  
In this paper, our main focus is on the $R$-module structure of
$\calS^\br(\Sigma)$; in particular, our main objective is the computation of
its Hilbert function, that is, $\dim\calS^\br_d(\Sigma)$ for $d\ge 0$.
Although the module structure is often subtle, it can be determined in certain
cases, as we illustrate in \Cref{ex:1}.

\begin{example}\label{ex:1}	

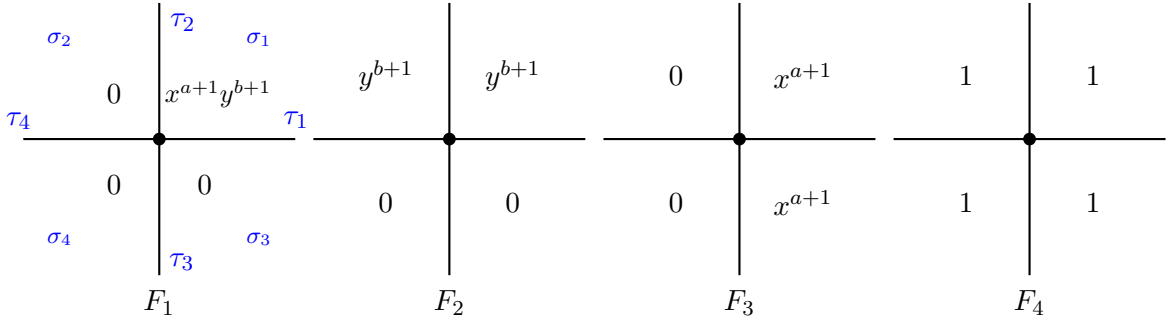
\begin{figure}
\begin{tikzpicture}[scale=1.2]
    \begin{scope}[xshift=0cm]
        \draw[thick] (-1.5,0) -- (1.5,0) 
            node[above right] {\hspace{-0.3cm}\color{blue}$\tau_1$};
        \node at (-1.55, 0.2) {\color{blue}$\tau_4$};
        \draw[thick] (0,-1.5) -- (0,1.5) 
            node[below right] {\color{blue}$\tau_2$};
        \node at (0.25, -1.35) {\color{blue}$\tau_3$};
        \node at (0.0, -1.8) {$F_1$};
        \node[font=\footnotesize] at (-1.1, 1.1) {\color{blue}$\sigma_2$};
        \node at (-0.5, 0.5) {$0$};
        \node[font=\footnotesize] at (1.1, 1.1) {\color{blue}$\sigma_1$};
        \node[font=\small] at (0.65, 0.5) {$x^{a+1}y^{b+1}$};
        \node[font=\footnotesize] at (-1.1, -1.1) {\color{blue}$\sigma_4$};
        \node at (-0.5, -0.5) {$0$};
        \node[font=\footnotesize] at (1.1, -1.1) {\color{blue}$\sigma_3$};
        \node at (0.5, -0.5) {$0$};
        \fill (0,0) circle (2pt);
    \end{scope}
    \begin{scope}[xshift=3.2cm]
        \draw[thick] (-1.5,0) -- (1.5,0);
        \draw[thick] (0,-1.5) -- (0,1.5);
        \node at (0.0, -1.8) {$F_2$};
        \node at (-0.7, 0.7) {$y^{b+1}$};
        \node at (0.7, 0.7) {$y^{b+1}$};
        \node at (-0.7, -0.7) {$0$};
        \node at (0.7, -0.7) {$0$};
        \fill (0,0) circle (2pt);
    \end{scope}
    \begin{scope}[xshift=6.4cm]
        \draw[thick] (-1.5,0) -- (1.5,0);
        \draw[thick] (0,-1.5) -- (0,1.5);
        \node at (0.0, -1.8) {$F_3$};
        \node at (-0.7, 0.7) {$0$};
        \node at (0.7, 0.7) {$x^{a+1}$};
        \node at (-0.7, -0.7) {$0$};
        \node at (0.7, -0.7) {$x^{a+1}$};
        \fill (0,0) circle (2pt);
    \end{scope}
    \begin{scope}[xshift=9.6cm]
        \draw[thick] (-1.5,0) -- (1.5,0);
        \draw[thick] (0,-1.5) -- (0,1.5);
        \node at (0.0, -1.8) {$F_4$};
        \node at (-0.7, 0.7) {$1$};
        \node at (0.7, 0.7) {$1$};
        \node at (-0.7, -0.7) {$1$};
        \node at (0.7, -0.7) {$1$};
        \fill (0,0) circle (2pt);
    \end{scope}
\end{tikzpicture}
\caption{Splines $F_1,F_2,F_3,F_4$ on the fan $\Sigma^\A$ 
of \Cref{ex:1}, with four $2$-dimensional cones $\sigma_1,\sigma_2,
\sigma_3,\sigma_4$ and codimension-one faces $\tau_1,\tau_2,\tau_3,
\tau_4$, as labelled in $F_1$.}
\label{splines2}
\end{figure}
		Let $R=\R[x,y]$, and consider the fan $\Sigma=\Sigma^\A \subseteq \mathbb{R}^2$ of the arrangement $\A$ which is the union of the $x$ and $y$ axes in $\RR^2$ (see \Cref{splines2}).  The arrangement $\A$ is both essential and generic (as is the fan $\Sigma^\A$).  We have
        ${\ell}_{\tau_1} = {\ell}_{\tau_4} = y$, and ${\ell}_{\tau_2} = {\ell}_{\tau_3} = x$.
		We also have
		$\Sigma_2 = \{\sigma_i\}_{i=1}^4$, $\Sigma_1 = \Sigma_1^0 = \{\tau_i\}_{i=1}^4$, and 
		$\Sigma_0 = \Sigma_0^0 = \{(0,0)\}$.
		
		For integers $a,b\geq -1$, we take the smoothness distribution $\br\colon\Sigma_1^\circ \rightarrow \mathbb{Z}_{\geq -1}$ by $\br(\tau_1) = \br(\tau_4) = b$, $\br(\tau_2) = \br(\tau_3) = a$.
        If we define $F_1, F_2, F_3, F_4 \in R^4$ as in Figure \ref{splines2}, these are splines in $\calS^r(\Sigma)$ of degrees $a+b+2$, \, $b+1$, \, $a+1$, and  $0$, respectively. 
    As an example of our notation, 
    $$F_2 = ((F_2)_{\sigma_1},(F_2)_{\sigma_2}, (F_2)_{\sigma_3}, (F_2)_{\sigma_4}) = (y^{b+1}, y^{b+1},0,0).$$  
    Note we have $(F_2)_{\sigma_1} - 
    (F_2)_{\sigma_2} = y^{b+1}-y^{b+1} = 0 \in 
    \langle \ell_{\tau_2}^{r(\tau_2)+1} \rangle = \langle x^{a+1} \rangle$
     where $\tau_2 = \sigma_1 \cap \sigma_2$, and $(F_2)_{\sigma_2} - (F_2)_{\sigma_4}= y^{b+1} - 0 \in \langle \ell_{\tau_4}^{r(\tau_4)+1}\rangle = \langle y^{b+1} \rangle$  where
     $\tau_4 = \sigma_2 \cap \sigma_4$.  The reader can check
     the other remaining conditions.
    These splines can be arranged in matrix form as
		\[
		\kbordermatrix{
			& F_1 & F_2 & F_3 & F_4 \\
			\sigma_1 &	x^{a+1}y^{b+1} & y^{b+1} & x^{a+1} & 1 \\
			\sigma_2 &	0 & y^{b+1} & 0 & 1 \\
			\sigma_3 &	0 & 0 & x^{a+1} & 1 \\
			\sigma_4 &	0 & 0 & 0 & 1
		}= M.
		\]
		Notice $\det (M)= (x^{a+1}y^{b+1})(x^{a+1})( y^{b+1}) = \prod_{\tau \in \Sigma_1^\circ} \ell_\tau^{\br(\tau)+1}$.  A criterion of Rose~\cite[Theorem~2.3]{Rose-Module-Bases} implies that $F_1,F_2,F_3,$ and $F_4$ form a free basis for $\calS^{\br}(\Sigma)$ as an $R$-module.  Thus
	\[
        \calS^\br(\Sigma) \cong {R}F_1 \oplus {R}F_2 \oplus {R}F_3 \oplus {R}F_4\cong {R}(-a-b-2) \oplus {R}(-b-1) \oplus {R}(-a-1) \oplus {R}.
        \]
        Note that in this case, the Hilbert function of $\calS^{\br}(\Sigma)$
        is straightforward to compute:
        $$\dim \calS^{\br}_d(\Sigma) = 
        \binom{-a-b-2+d+1}{1} + \binom{-b-1+d+1}{1} + \binom{-a-1+d+1}{1}
        + \binom{d+1}{1}.$$
        If $d \ge a+b+1$, we have
        $\dim \calS^{\br}_d(\Sigma) = 4d-2a-2b.$  Thus $\HP_{\calS^{\br}(\Sigma)}(d)=4d-2a-2b$.
	\end{example}

In this paper, we will restrict ourselves to smoothness distributions $\br:\Sigma^\A_{m-1}\to\ZZ_{\ge -1}$ that are \textit{constant along hyperplanes}, in the following sense.

\begin{definition}
\label{def:Constant_Along_Hyperplanes}
Let $\A\subseteq\RR^m$ be a hyperplane arrangement and let $\Sigma^\A$ 
the fan of $\A$.
We say that a smoothness distribution $\br:\Sigma^\A_{m-1}\to \ZZ_{\ge -1}$ is \emph{constant along hyperplanes} if, whenever $\tau_1,\tau_2\in\Sigma^\A_{m-1}$ have the same linear span, then $\br(\tau_1)=\br(\tau_2)$.

Since the linear span of faces $\tau\in\Sigma^\A_{m-1}$ is always a hyperplane in $\A$, such a distribution $\br$ naturally descends to a map $\tilde{\br}:\A\to\ZZ_{\ge -1}$, where $\tilde{\br}(H)=\br(\tau)$ for any $\tau\in\Sigma^\A_{m-1}$ whose linear span is $H$.

We abuse notation by replacing $\tilde{\br}$ with $\br:\A\to\ZZ_{\ge-1}$ to indicate that we are using a smoothness distribution which is constant along hyperplanes.
\end{definition}

\Cref{ex:1} exemplifies such a smoothness distribution.

\subsection{Reduction to the essential case}

If $\Sigma\subset \RR^m$ is a fan which is \textit{not} essential, then there is a non-trivial linear space $V$ of dimension $m-i$ ($i\ge 1$) which is contained in all of the cones of $\Sigma$.  Changing coordinates if necessary, we may assume that $V$ is defined by the vanishing of the variables $x_1,\ldots, x_i$.  Let $\pi:\RR^m\to\RR^{i}$ denote the orthogonal projection of $\RR^m$ onto $\RR^{i}$ defined in coordinates by $\pi(x_1,\ldots,x_m)=(x_1,\ldots,x_{i})$.  Since $V\subset \sigma$ for every cone $\sigma\in\Sigma$, the collection of cones $\pi(\Sigma)=\{\pi(\sigma):\sigma\in\Sigma\}$ is a fan in $\RR^{i}$.  Moreover, $\sigma\to\pi(\sigma)$ is a bijection between $j$-dimensional cones $\sigma\in\Sigma_j$ and $(j-(m-i))$-dimensional cones $\pi(\Sigma)_{j-(m-i)}$, for $m-i\le j\le m$, with inverse $\sigma'\to \pi^{-1}(\sigma')$.  It follows that a smoothness distribution $\br:\Sigma^\circ_{m-1}\to \RR$ naturally induces a smoothness distribution $\br':\pi(\Sigma)^\circ_{i-1}\to\RR$ via $\br'=\br\circ\pi^{-1}$.  We have the following structural result.

\begin{proposition}\label{prop:NonEssentialTensorProduct}
Let $\Sigma\subseteq\RR^m$ be a pure $m$-dimensional fan which is not essential.  Without loss of generality, suppose that every cone $\sigma\in\Sigma$ contains the vector space $V$ defined by the vanishing of $x_1,\ldots,x_i$ and let $\pi:\RR^m\to\RR^{i}$ be the orthogonal projection defined by $\pi(x_1,\ldots,x_m)=(x_1,\ldots,x_{i})$.  Then
\begin{align*}
 S^\br(\Sigma)\cong   \calS^{\br'}(\pi(\Sigma))\otimes_{\RR} &\RR[x_{i+1},\ldots,x_m],\\
\dim S^\br_d(\Sigma)= \sum_{j=0}^d  \binom{j+m-i}{m-i}\dim \calS^{\br'}_{d-j}(\pi(\Sigma)),& \quad \mbox{and }\quad 
\hilb_{\calS^\br(\Sigma)}(t)= \frac{\hilb_{\calS^{\br'}(\pi(\Sigma))}(t)}{(1-t)^{m-i}}.
\end{align*}
\end{proposition}
\begin{proof}
The proof that $S^\br(\Sigma)\cong S^\br(\pi(\Sigma))\otimes_{\RR} \RR[x_{i+1},\ldots,x_m]$ is essentially the same as the proof of \cite[Proposition~3.1]{SSVY2024}, so we only give a sketch.  Take any spline $F\in \calS^{\br'}(\pi(\Sigma))$ and pull it back via $\pi$ to a function $F\circ \pi:\RR^m\to \RR$.  One then checks that $F\circ\pi\in \calS^{\br}(\Sigma)$.  Then define the map
\[
\phi: \calS^{\br'}(\pi(\Sigma))\otimes_{\RR} \RR[x_{i+1},\ldots,x_m]\to \calS^{\br}(\Sigma)
\]
on simple tensors by $\phi(F\otimes p)=(F\circ \pi)\cdot p$.  This is an injective map, and one needs to show that $\phi$ is surjective.  This can be done by considering the constituents of a spline $G\in \calS^{\br}(\Sigma)$ as polynomials in $\RR[x_1,\ldots,x_i]$ with coefficients in $\RR[x_{i+1},\ldots,x_{m}]$.  One shows that the terms of each constituent with a fixed coefficient fit together to make a spline in $\calS^{\br}(\Sigma)$ which is the image of a simple tensor under $\phi$, establishing surjectivity.  See the proof of \cite[Proposition~3.1]{SSVY2024} for further details.  The equations for the Hilbert function and Hilbert series of $S^\br(\Sigma)$ follow from the first isomorphism.
\end{proof}

\begin{remark}
Suppose that $\Sigma\subseteq\RR^3$ is a pure, three-dimensional fan.  If $\Sigma$ is not essential, then we may assume that every cone of $\Sigma$ contains the line defined by the vanishing of $x_1$ and $x_2$.  Let $\pi:\RR^3\to\RR^2$ be the projection onto the first two coordinates.  Then $\pi(\Sigma)\subseteq \RR^2$ and $S^\br(\Sigma)\cong \calS^{\br'}(\pi(\Sigma))\otimes_\RR \RR[x_3]$ by \Cref{prop:NonEssentialTensorProduct}.  The dimension of spline spaces on fans in $\RR^2$ are completely determined (see~\cite{FatPoints}), and thus the dimension of the spline space $\calS^{\br}(\Sigma)$ is known if $\Sigma$ is not essential.  Thus in our work we will only consider \textit{essential} fans $\Sigma\subseteq \RR^3$.
\end{remark}
    
\section{Homological methods for splines}
\label{sec.homologicalmethodsforsplines}

In this section we recall homological methods for splines that were introduced by Billera~\cite{Homology} and were later refined in an important way by Schenck and Stillman~\cite{LCoho}.  We also prove a presentation for a particular homology module that is crucial for the remainder of the paper.

\subsection{Chain complexes for splines}
Throughout this section we use the conventions outlined in \Cref{ss:indexing_conventions_direct_sums}.  Let $\Sigma\subseteq\RR^m$ be a pure, hereditary, $m$-dimensional fan.  Denote by $\calR[\Sigma]_{\bullet}$ the chain complex
\[
\calR[\Sigma]_{\bullet}\colon\quad  0\longrightarrow R^{\Sigma_m}\xrightarrow{\partial_m} R^{\Sigma^\circ_{m-1}} \xrightarrow{\partial_{m-1}} \cdots\xrightarrow{\partial_2} R^{\Sigma^\circ_1} \xrightarrow{\partial_1} R^{\Sigma_0^\circ} \longrightarrow 0\ ,
\]
which is the cellular chain complex (relative to its boundary) of the finite CW complex $X$ obtained by intersecting $\Sigma$ with a solid disc of radius $1$. 
 We shall also use the notation $\calR_i=R^{\Sigma_i}$.
 The matrices representing the differentials $\partial_i$ have entries which are all $0,1,$ or $-1$.  As $\Sigma$ is not compact, some care must be taken here -- the homology modules of $\calR[\Sigma]_\bullet$ compute the cellular homology of the CW complex $X$ relative to its boundary with coefficients in $R$ (the so-called \textit{Borel-Moore} homology of $\Sigma$ relative to its boundary).  We will need an explicit description of the cellular boundary maps only when $m=3$ in the proof of \Cref{lem:H1Jpres}, and we will explicitly describe them there.  Readers may refer to~\cite{AssHom} for a general description of these boundary maps.

Let $\Sigma\subset\RR^m$ be a complete and essential fan.  For each $\tau \in \Sigma_{m-1}$, let ${\ell}_\tau \in \mathbb{R}$ be a linear form vanishing on $\tau$, and define
$J(\tau) = \langle {\ell}_\tau^{\br(\tau)+1} \rangle$.  
For all $\beta \in \Sigma_i$, where $0 \leq i \leq m-1$, define 
$$J(\beta) = \sum_{\substack{\tau\in\Sigma_{m-1}\\\tau \supseteq \beta}} J(\tau).$$
Collecting these ideals yields a subcomplex of $\calR[\Sigma]_\bullet$ denoted $\calJ[\Sigma]_\bullet$, given by
\[
\calJ[\Sigma]_\bullet\colon \quad  0\longrightarrow \bigoplus_{\sigma\in\Sigma_m} J(\sigma)=0 \rightarrow \bigoplus_{\tau\in\Sigma^\circ_{m-1}} J(\tau)\xrightarrow{\partial_{m-1}}\cdots\xrightarrow{\partial_2} \bigoplus_{\beta\in\Sigma^\circ_1} J(\beta) \xrightarrow{\partial_1} J(\cv) \longrightarrow 0\ .
\]
The chain complex introduced by Billera in \cite{Homology} and modified by Schenck and Stillman in \cite{LCoho} is the quotient of the complex $\calR[\Sigma]_\bullet$ by $\calJ[\Sigma]_{\bullet}$, which we denote by $\calR/\calJ[\Sigma]_\bullet$:
\begin{equation}\label{eq:BBScomplex}
	\calR/\calJ[\Sigma]_\bullet\colon 0\rightarrow\bigoplus_{\sigma \in \Sigma_m} {R} \xrightarrow{\bar\partial_m} \bigoplus_{\tau \in \Sigma^\circ_{m-1}} {R}/J(\tau) \xrightarrow{\bar\partial_{m-1}} \cdots\xrightarrow{\bar\partial_2} \bigoplus_{\beta\in\Sigma^\circ_1} R/J(\beta) \xrightarrow{\bar\partial_1} \dfrac{R}{J(\cv)} \rightarrow 0. 
\end{equation}	
If the fan $\Sigma$ is fixed, we will write $\calR_\bullet$, $\calJ_\bullet$, and $\calR/\calJ_\bullet$ to denote the chain complexes $\calR[\Sigma]_\bullet$, $\calJ[\Sigma]_\bullet$, and $\calR/\calJ[\Sigma]_\bullet$, respectively.  We will sometimes write
$\calJ^\br$ if we want to highlight the specific smoothness
condition we are considering. 

A key observation here is that \Cref{def:splines} implies that the top homology module of the chain complex \cref{eq:BBScomplex} satisfies
$\HH_m(\calR/\calJ_\bullet) \cong S^\br(\Sigma)$.
 We summarize some well-known properties of $\calR/\calJ[\Sigma]_{\, \bullet}$ when $\Sigma\subset\RR^3$ is a fan (see for instance \cite{Spect, LCoho,Deepesh-Nelly-supers,Michael-Nelly-SIAGA}).

\begin{proposition}\label{prop:FrequentlyUsedIsomorphisms}
	If $\Sigma\subseteq\R^3$ is a fan with smoothness distribution $\br$, then the following hold: 
    \begin{enumerate}[label=({\arabic*})]
        \item $\calS^\br(\Sigma)\cong \HH_3(\calR/\calJ_\bullet)\cong R\oplus \HH_{2}(\calJ_\bullet)$;
        \item $\HH_1(\calR/\calJ_\bullet)=\HH_0(\calR/\calJ_\bullet)=0$;
        \item $\HH_i(\calR/\calJ_\bullet)\cong \HH_{i-1}(\calJ_\bullet)$ for $i=1,2$. 
    \end{enumerate}
\end{proposition}

The Hilbert function of $\calS^{\bm r}(\Sigma)$ can be computed from the Hilbert functions of the terms of $\calJ_\bullet$ and its lower homologies (this was Billera's crucial insight in \cite{Homology}).  We record the three-dimensional version of this observation, which is sufficient for our purposes.  The proof is straightforward homological algebra; see e.g. \cite[Proposition 2.6]{Michael-Nelly-SIAGA}.

\begin{proposition}\label{prop:EulerCharacteristicAndDimension}
If $\Sigma\subseteq\R^3$ is a complete fan with smoothness distribution $\br$, then
	\begin{align*}
		\dim \calS^\br_d(\Sigma)
        & = 2\dim R_d+ \sum_{\tau\in\Sigma_2}\dim J(\tau)_d -\sum_{\gamma\in\Sigma_1}\dim J(\gamma)_d-\dim (R/J(\cv))_d+\dim \HH_1(\calJ_\bullet)_d.
	\end{align*}
In particular,
\[
\dim \calS^r_d(\Sigma)\ge 2\dim R_d+ \sum_{\tau\in\Sigma_2}\dim J(\tau)_d -\sum_{\gamma\in\Sigma_1}\dim J(\gamma)_d-\dim (R/J(\cv))_d.
\]
\end{proposition}

\subsection{A presentation for $\HH_1(\calJ_\bullet)$ for three-dimensional hyperplane arrangements}\label{ss:H1arrangements} 

In this section, we prove a presentation of $\HH_1(\calJ_\bullet)$ for the fan of an essential three-dimensional hyperplane arrangement.  We use the following presentation for $\HH_1(\calJ_\bullet)$ from \cite[Lemma~9.12]{AssHom}.

\begin{lemma}\label{lem:H1Jpres}
	Let $\Sigma\subseteq\R^3$ be a
    complete and essential three-dimensional  fan.
    Define the following  $R$-submodules of $\bigoplus\limits_{\tau\in\Sigma_2} Re_\tau$:
	\begin{gather*}
V^{\br}_{\Sigma}=\biggl\lbrace\sum_{\tau\in\Sigma_2} f_\tau e_\tau \in \bigoplus\limits_{\tau\in\Sigma_2} Re_{\tau}\colon \sum_{\tau\in\Sigma_2} f_\tau \ {\ell}_\tau^{\br(\tau)+1}=0\biggr\rbrace \quad\text{and} \\
V^{1,\br}_{\Sigma} = \biggl\{\sum_{\tau\in\Sigma_2} f_\tau e_\tau \in \bigoplus\limits_{\tau\in\Sigma_2} Re_{\tau}\colon 
\sum_{\tau\in\Sigma_2, \tau\supseteq\gamma}  f_\tau \, {\ell}_\tau^{\br(\tau)+1}=0 \, \text{\ for each }\ \gamma\in\Sigma_1\biggr\}.
	\end{gather*}
Then, $V^{1,\br}_{\Sigma}\subseteq V^{\br}_{\Sigma}$ and $\HH_1(\calJ[\Sigma]_\bullet)\cong V^{\br}_{\Sigma}/V^{1,\br}_{\Sigma}$ as $R$-modules. 
\end{lemma}
\begin{proof}
For each ray $\gamma\in\Sigma_1$, define the module of relations around $\gamma$ as
\[V^{\br}_\gamma 
= 
\Big\{\sum_{\tau\in\Sigma_2, \tau \supseteq \gamma} 
f_\tau e_{\tau,\gamma}\in \bigoplus_{\tau\in\Sigma_2, \tau\supseteq \gamma} R e_{\tau,\gamma} \colon \sum_{\tau\in\Sigma_2,\tau\supseteq \gamma}f_\tau \ell_\tau^{\br(\tau)+1} 
= 
0 \Big\} .
\]
If $J(\cv)$ denotes the ideal of the central vertex of $\Sigma$, we construct the diagram in Figure \ref{fig:H1Jpres}, where the top row is the complex $\calJ_\bullet[\Sigma]$. The diagram commutes and the columns are exact.
\begin{figure}[htp]
    \centering
	\begin{tikzcd}
		0 & 0 & 0\\
		\bigoplus\limits_{\tau\in\Sigma_2} J(\tau)b_\tau \ar{r}{\delta_2}\ar{u} & \bigoplus\limits_{\gamma\in\Sigma_1} J(\gamma) b_\gamma \ar{r}{\delta_1}\ar{u} & J(\cv) \ar{u}\\
		\bigoplus\limits_{\tau\in\Sigma_2} Re_\tau \ar{r}{\delta'_2}\ar{u}{q_2} & \bigoplus\limits_{\gamma\in\Sigma_1}\bigoplus\limits_{\substack{\tau\in\Sigma_2\\ \gamma\subseteq\tau}} Re_{\tau,\gamma} \ar{r}{\delta'_1}\ar{u}{q_1} & \bigoplus\limits_{\tau\in\Sigma_2} Re_{\tau}\ar{u}{q_0} \\
		0\ar{u}\ar{r} & \bigoplus\limits_{\gamma\in\Sigma_1} V^{\br}_\gamma \ar{u}\ar{r}{\iota} & V^{\br}_{\Sigma} \ar{u}\\
		& 0\ar{u} & 0\ar{u} 
	\end{tikzcd}
    \caption{Commutative diagram for the proof of Lemma~\ref{lem:H1Jpres}.}
    \label{fig:H1Jpres}
\end{figure}

    The maps $q_0,q_1,q_2$ in Figure~\ref{fig:H1Jpres} are defined on basis symbols by $q_0(e_\tau)={\ell}_\tau^{\br(\tau)+1}$, $q_1(e_{\tau,\gamma})={\ell}_{\tau}^{\br(\tau)+1}b_{\gamma}$, and $q_2(e_{\tau})={\ell}_\tau^{\br(\tau)+1}b_\tau$, and extended linearly.  Observe that, by definition, $V^{\br}_{\Sigma}$ is the kernel of $q_0$ and $\bigoplus_{\gamma\in\Sigma_1} V_{\gamma}^{\br}$ is the kernel of $q_1$.  Clearly $q_2$ is an isomorphism.  Thus the columns are exact.

    The maps $\delta_1$ and $\delta_2$ in Figure~\ref{fig:H1Jpres} are cellular boundary maps, and thus depend on a choice of orientation for each face of $\Sigma$ (the homology of the chain complex is independent of the choice of orientation in the sense that different orientations yield isomorphic homology modules).  For concreteness, we orient each ray towards the origin so that $\delta_1(b_\gamma)=1$ for all $\gamma\in\Sigma_1$.  That is, $\delta_1$ just takes the sum of the components: $\delta_1(\sum_{\gamma\in\Sigma_1}f_{\gamma}b_\gamma)=\sum_{\gamma\in\Sigma_1} f_{\gamma}$.  Furthermore, observe that each two-dimensional face $\tau\in\Sigma_2$ contains two rays $\gamma_1,\gamma_2\in\Sigma_1$.  For each $\tau\in\Sigma_2$ we will choose one of these to be the \textit{terminal} ray, written $t(\tau)$, and one of these to be the \textit{initial} ray, written $i(\tau)$.  With this convention and the convention for $\delta_1$, we may take the map $\delta_2$ to be $\delta_2(b_{\tau})=b_{t(\tau)}-b_{i(\tau)}$ on basis symbols, and extended in an $R$-linear fashion.

    We now have no choice in defining the maps $\delta'_1,\delta'_2$ if we want the diagram in Figure~\ref{fig:H1Jpres} to commute.  That is, for a fixed $\tau\in\Sigma_2$ and $\gamma\in\Sigma_1$ so that $\gamma\subseteq \tau$ we must have $\delta'_1(e_{\tau,\gamma})=e_{\tau}$ and $\delta'_2(e_{\tau})=e_{\tau,t(\tau)}-e_{\tau,i(\tau)}$.  Thus we see that the middle row is in fact exact; the map $\delta'_2$ has the effect of gluing together the basis symbols $e_{\tau,i(\tau)}$ and $e_{\tau,t(\tau)}$ that correspond to different rays in $\Sigma_1$ that are contained in $\tau$, leaving just a copy of $e_{\tau}$ for every codimension one face $\tau\in\Sigma_2$ in the cokernel.  

    Since the middle row in Figure~\ref{fig:H1Jpres} is exact and the map $\delta'_1$ is surjective, the long exact sequence in homology yields the isomorphisms $\HH_2(\calJ^\br[\Sigma]_\bullet)\cong \mbox{ker}(\iota)$ and $\HH_1(\calJ^\br[\Sigma]_\bullet)\cong \mbox{coker}(\iota)$.  The image of $\bigoplus\limits_{\gamma\in\Sigma_1} V^{\br}_\gamma$ under $\iota$ is precisely $V^{1,\br}_{\Sigma}$, so we are done.
\end{proof}

Now suppose $\A=\{H_1,\ldots,H_k\} \subseteq\RR^3$ is an essential hyperplane arrangement and $\Sigma^\A$ is the fan of $\A$.  Recall that the fan of a hyperplane arrangement is always complete.  Let $\br:\A\to \ZZ_{\ge -1}$ be a smoothness distribution (see \Cref{def:Constant_Along_Hyperplanes}).  Our next objective is to prune the presentation for $H_1(\calJ^\br[\Sigma])$ given in \Cref{lem:H1Jpres} when $\Sigma=\Sigma^\A$.  This will prepare us to make the connection to Koszul homology in \Cref{sec:KoszulConnection}.   To this end, let
\begin{equation}\label{eq.V_Ardefn}
V^{\br}_\A=\Big\{\sum_{H\in\A} f_He_H: \sum_{H\in\A}f_H{\ell}_H^{\br(H)+1}=0\Big\}\subseteq \bigoplus_{H\in\A} Re_H
\end{equation}
be the syzygy module of the forms $\{\ell_H^{\br(H)+1}:H\in\A\}$.

Recall that $\calL_2(\A)$ is the set of all lines appearing as the intersection of at least two hyperplanes of $\A$.  We define $V_\A^{1,\br}$ to be the $R$-submodule of $\bigoplus_{H\in \A} Re_H$ generated by
\begin{equation}\label{eq.V_Ar-1defn}
\Big\{\sum_{H\supseteq W} f_He_H: W\in\calL_2(\A), \sum_{H\supseteq W} f_H{\ell}_H^{\br(H)+1}=0\Big\}.
\end{equation}
Clearly $V_\A^{1,\br}\subseteq V_{\A}^\br$.

\begin{lemma}\label{lem:PruningH1Pres}
Let $\A\subseteq\RR^3$ be an essential hyperplane arrangement and $\Sigma=\Sigma^\A$ the fan of $\A$.  Then
$
{V^{\br}_\A}/{V^{1,\br}_\A}\cong {V^{\br}_{\Sigma}}/{V^{1,\br}_\Sigma}.
$
In particular, $\HH_1(\calJ[\Sigma^\A]_\bullet)\cong {V^{\br}_\A}/{V^{1,\br}_\A}$.
\end{lemma}

\begin{proof}
 The final statement is immediate from Lemma~\ref{lem:H1Jpres}, so we prove the isomorphism ${V^{\br}_\A}/{V^{1,\br}_\A}\cong {V^{\br}_{\Sigma}}/{V^{1,\br}_\Sigma}.$  Consider the $R$-module homomorphism 
$\phi: \bigoplus\limits_{\tau\in \Sigma^\A_2} Re_\tau\to \bigoplus\limits_{H\in\A}Re_H$ defined by $\phi(e_\tau)=e_H$ if $\tau\subseteq H$.  This is well-defined since a codimension one cone $\tau$ can only be contained in one hyperplane of $\A$. 
 In the remainder of this proof we let $H_\tau$ be the unique hyperplane in $\A$ which contains $\tau$.  We also have that $\phi$ is a graded map since we assume $\br$ is constant along hyperplanes.

 We first claim that $\phi(V^{\br}_\Sigma)=V^{\br}_\A$.  Clearly if $\sum f_\tau e_\tau \in V^{\br}_\Sigma$ then 
 $\phi(\sum\limits_{\tau\in\Sigma_2} f_\tau e_\tau)=$ $\sum\limits_{\tau\in\Sigma_2}f_\tau e_{H_\tau}$ is in $V^\br_{\A}.$
 So $\phi(V^\br_\Sigma)\subseteq V^\br_\A$.  Now suppose that $\sum_{H\in\A} f_H e_H\in V^{\br}_\A$.  Then, by definition, $\sum_{H\in\A} f_H{\ell}_H^{\br(H)+1}=0$.  Now we may take any collection $\{\tau_H\}_{H\in \A}\subseteq \Sigma_2$ so that $\tau_H\subseteq H$.  Then $\sum\limits_{H\in \A} f_He_{\tau_H} \in \bigoplus\limits_{\tau\in\Sigma_2} Re_\tau$ and $\sum f_H {\ell}_{\tau_H}^{\br(\tau_H)+1}=\sum f_H {\ell}_{H}^{\br(H)+1}=0$.  So $\phi(V^\br_\Sigma)=V^\br_\A$.

Next we claim that $\phi(V^{1,\br}_{\Sigma})=V^{1,\br}_\A$.  First suppose that $\sum\limits_{\tau\supseteq\gamma} f_\tau e_\tau\in V^{1,\br}_\Sigma$ is one of the $R$-module generators of $V^{1,\br}_\Sigma$.  Then $\phi(\sum\limits_{\tau\supseteq\gamma} f_\tau e_\tau)=\sum\limits_{\tau\supseteq\gamma} f_\tau e_{H_\tau},$ and $\sum f_\tau{\ell}_{H_{\tau}}^{\br(H_\tau)+1}=\sum f_\tau{\ell}_{\tau}^{\br(\tau)+1}=0,$ and the hyperplanes $H_\tau$ all intersect along the line $W$ which is the span of the ray $\gamma$.  So $\phi(\sum\limits_{\tau\supseteq\gamma} f_\tau e_\tau)\in V^{1,\br}_{\A}.$  Now suppose that $W\in\calL_2(\A)$ and $\sum\limits_{H\supseteq W} f_H e_H\in V^{1,\br}_\A$ is one of the $R$-module generators of $V^{1,\br}_\A$.  Since $\Sigma^\A$ is the fan induced by $\A$, there is a ray $\gamma\in \Sigma^\A_1$ so that the span of $\gamma$ is $W$ and for every $H\in\A$ so that $W\subseteq H$ there exists a $\tau_H\in\Sigma_2$ which contains $\gamma$ and whose span is $H$.  Then observe that 
$
\sum\limits_{\tau_H\supseteq\gamma} f_H {\ell}_{\tau_H}^{\br(\tau_H)+1}=\sum\limits_{\tau_H\supseteq\gamma} f_H {\ell}_{H}^{\br(H)+1}=0,
$
so $\sum\limits_{\tau_H\supseteq \gamma} f_H e_{\tau_H}\in V^{1,\br}_\Sigma$.  Clearly $\phi(\sum\limits_{\tau_H\supseteq \gamma} f_H e_{\tau_H})=\sum\limits_{H\supseteq W} f_H e_H$.  So $\phi(V^{1,\br}_{\Sigma})=V^{1,\br}_\A$.

Now consider $\ker(\phi)$ -- clearly this is generated as an $R$-module by the differences
\[
D_1=\{e_{\tau}-e_{\tau'}: \tau,\tau'\subseteq H\mbox{ for some } H\in \A\}.
\]
Since ${\ell}_\tau={\ell}_{\tau'}$ if $\tau,\tau'$ are both subsets of the same hyperplane, $e_{\tau}-e_{\tau'}\in V^{\br}_{\A}$ whenever $e_{\tau}-e_{\tau'}\in \ker(\phi)$.  Since these span $\ker(\phi)$, it follows that $\ker(\phi)\subseteq V^{\br}_{\Sigma}$.  

Now we claim that $\ker(\phi)\subseteq V^{1,\br}_{\Sigma}$ also.  For this we observe that $\ker(\phi)$ is also generated as an $R$-module by the differences
\[
D_2=\{e_{\tau}-e_{\tau'}: \tau,\tau'\subseteq H\mbox{ for some } H\in \A \mbox{ and } \tau\cap\tau'=\gamma\in\Sigma^\A_1\}.
\]
To prove this, suppose that $\tau,\tau'\in\Sigma^\A_2$ so that $\tau,\tau'$ are both contained in the same hyperplane $H\in\A$.  Then $\tau,\tau'$ are chambers of the hyperplane arrangement $\A^H$, which is the hyperplane arrangement in the linear space $H\cong \RR^2$ formed by $H'\cap H$ for every $H'\in\A$ with $H'\neq H$.  It follows that there exist $\tau=\tau_1,\tau_2,\ldots,\tau_k=\tau'\in\Sigma^\A_2$ so that $\tau_1,\ldots,\tau_k$ are all in $H$ and $\tau_i, \tau_{i-1}$ are adjacent chambers in $\A^H$ for $i=2,\ldots,k$, and hence $\tau_{i}\cap\tau_{i-1}=\gamma_i\in\calL_2(\A)$.  Thus $e_\tau-e_{\tau'}=\sum_{i=1}^{k-1} (e_{\tau_i}-e_{\tau_{i+1}})$, and thus $D_2$ also generates $\ker(\phi)$ as an $R$-module.  Since $D_2\subseteq V^{1,\br}_{\Sigma}$, it follows that $\ker(\phi)\subseteq V^{1,\br}_\Sigma$ also.  
Putting everything together, we have
\[
\dfrac{V^{\br}_\Sigma}{V^{1,\br}_{\Sigma}}\cong \dfrac{V^{\br}_{\Sigma}/\ker(\phi)}{V^{1,\br}_{\Sigma}/\ker(\phi)}\cong \dfrac{\phi(V^{\br}_{\Sigma})}{\phi(V^{1,\br}_{\Sigma})}\cong\dfrac{V^\br_\A}{V^{1,\br}_\A}. \qedhere
\]
\end{proof}

\section{Connection to Koszul Homology}\label{sec:KoszulConnection}

In this section we use \Cref{lem:PruningH1Pres} to relate the homology module $\HH_1(\calJ[\Sigma^\A]_\bullet)$ to a standard tool in commutative algebra, namely \textit{Koszul homology}. We first recall the definition of Koszul homology and a few of its salient properties.

Let $\kk$ be a field ($\kk=\mathbb R$ for our purposes), $R=\kk[x_1, \dots, x_m]$ and $\by=(y_1, \dots, y_n)$  in $R$. We let $I$ be the ideal generated by $\by$. 
Denote by $\K_{\bigcdot}(\by)$ the Koszul complex on $\by$ and by $\HK_{\bigcdot}(\by)$ its homology. To describe $\K_{\bigcdot}(\by)$, we let $F$ denote a free $R$-module with basis $e_1, \dots, e_n$ and we let $\varphi\colon F\to I$ denote the $R$-module homomorphism with $\varphi(e_i)=y_i$. Then we have $\K_i(\by)=\bigwedge^iF$ and the differential $\partial=\partial^{\K_{\bigcdot}(\by)}$ of $\K_{\bigcdot}(\by)$ is described for $z_1, \dots, z_i\in F$ by 
\[
\partial_i(z_1\wedge\cdots\wedge z_i)=\sum_{j=1}^{i}(-1)^{j+1} \varphi(z_j) z_1\wedge\cdots\wedge\widehat{z_j}\wedge \cdots \wedge z_i. 
\]
It is clear that $\HK_0(\by)=R/I$.  Furthermore, $\HK_1(\by)=\ker(\partial_1)/\text{Im}(\partial_2)$ can be described as follows. The module $\ker(\partial_1)$ is the {\it module of syzygies} of $\by$, namely
\begin{equation}\label{eq.koszulkernel}
\ker(\partial_1)=\ker(\varphi)=\Big\{\sum_{i=1}^n \alpha_ie_i\in F\colon \sum_{i=1}^n \alpha_iy_i=0\Big\}.
\end{equation}
The map $\partial_2\colon F\wedge F\to F$ is described for $z,z'\in F$ by 
$
\partial_2(z\wedge z')=\varphi(z)z'-\varphi(z')z,
$
and its image is generated by the elements of $F$ of the form $y_ie_j-y_je_i$ with $i,j\in [n]$, $i<j$, which are called {\it Koszul syzygies} of $\by$. 

When $\by$ is a minimal generating set of $I$, Koszul homology is independent, up to isomorphism, on the choice of $\by$ and in this case we also write $\HK_{\bigcdot}(I)=\HK_{\bigcdot}(\by)$. 

It is known, see \cite[Proposition 1.6.5(b)]{cohenMacaulay}, that the ideal $I$ annihilates $\HK_{\bigcdot}(\by)$. In particular, if $R/I$ is finite length, then $\HK_i(\by)$ is finite length for all $i$. 

\subsection{Koszul homology and generic three-dimensional hyperplane arrangements}

\begin{notation}\label{not:ellbr}
Let $\A = \bigcup_{i=1}^n H_i\subset \RR^3$ be a central hyperplane arrangement and $\br:\A \to \ZZ_{\ge 0}$ a smoothness distribution (c.f., \Cref{def:Constant_Along_Hyperplanes}).  We put $r_i=\br(H_i)$ and denote by $\bm{\ell}^{\br}$ the sequence $(\ell_{H_i}^{r_i+1}: 1\le i\le n)$ and by $\langle \bm{\ell}^{\br}\rangle$ the ideal generated by the $\bm{\ell}^{\br}$.
\end{notation}

Recall that a hyperplane arrangement $\A\subseteq\RR^3$ is called \textit{generic} if no three hyperplanes of $\A$ intersect along a line.   We now come 
to one of the main results of the paper that
shows a tight connection between $\calS^{\br}(\Sigma^\A)$ and Koszul homology.  In particular, we can now relate some of the homological groups that appear in Section
3 to the Koszul homology.

\begin{theorem}\label{thm.linkhomologoies}
Suppose $\A\subseteq\RR^3$ is an essential hyperplane arrangement and $\br:\A\to\ZZ_{\ge -1}$ is a smoothness distribution.  Then
\begin{enumerate}[label=({\arabic*})]
\item\label{thm.linkhomologoies-1} $\HH_1(\calJ^\br_\bullet)$ is a quotient of the first Koszul homology $\HK_1(\bm{\ell}^{\br})$  of $\bm{\ell}^{\br}$.
\item\label{thm.linkhomologoies-2} $\HH_1(\calJ^\br_\bullet)$ has finite length.
\item\label{thm.linkhomologoies-3} If $\A$ is generic, then $\HH_1(\calJ^\br_\bullet)\cong \HK_1(\bm{\ell}^{\br})$.
\end{enumerate}
\end{theorem}

\begin{proof}
We apply Koszul homology to the elements
of $\bm{\ell}^{\br}$.  In particular,
we have the following maps
$$  \cdots \longrightarrow \K_2(\bm{\ell}^{\br}) 
\stackrel{\partial_2}{\longrightarrow} \K_1(\bm{\ell}^{\br}) 
\stackrel{\partial_1}{\longrightarrow} 
I \longrightarrow 0.$$
We first prove \ref{thm.linkhomologoies-1}.
By \cref{lem:PruningH1Pres},
$\HH_1(\calJ^\br) = 
V_\A^\br/V_{\A}^{\br,1}$.
Since $\HK_1(\bm{\ell}^{\br}) = 
\ker(\partial_1)/\mbox{im}(\partial_2)$, it suffices to prove that  $\ker(\partial_1) = V_{\A}^\br$ and $\mbox{im}(\partial_2) \subseteq
V_{\A}^{\br,1}$.  By (\ref{eq.koszulkernel}), we have
$$\ker(\partial_1) =\Big\{\sum_{H \in \A} \alpha_He_H\in F\colon \sum_{H \in \A} \alpha_H\ell_H^{\br(H)+1}=0\Big\}.
$$
Note we are using $\A$ for our indexing
set of $\bm{\ell}^{\br}$ rather than $\{1,\ldots,n\}$.
But the set on the right is exactly how 
$V_\A^\br$ was defined in 
(\ref{eq.V_Ardefn}), thus giving 
$\ker(\partial_1)=V^\br_\A$.

The image $\mbox{im}(\partial_2)$ is generated by elements of the form 
$\ell_{H_2}^{\br(H_2)+1}e_{H_1} - 
\ell_{H_1}^{\br(H_1)+1}e_{H_2}$.  
Since $W=H_1\cap H_2\in \calL_2(\A)$, $\ell_{H_2}^{\br(H_2)+1}e_{H_1} - 
\ell_{H_1}^{\br(H_1)+1}e_{H_2}\in V_\A^{\br,1}$.  Thus $\mbox{im}(\partial_2)\subseteq V_\A^{\br,1}$.

For \ref{thm.linkhomologoies-2}, recall that $\langle \bm{\ell}^{\br}\rangle$ annihilates $\HK_1(\bm{\ell}^{\br})$.  Since $\A$ is essential, there are three linearly independent linear forms among $\{\ell_H: H\in\A\}$.  It follows that, up to a change of coordinates, we may assume that $\langle \bm{\ell}^{\br}\rangle$ has a power of $x$, a power of $y$, and a power of $z$ among its generators.  Hence a power of the maximal ideal $\langle x,y,z\rangle$ is contained in $\langle \bm{\ell}^{\br}\rangle$ and so annihilates $\HK_1(\bm{\ell}^{\br})$.  It follows that $\HK_1(\bm{\ell}^{\br})$ has finite length.

For \ref{thm.linkhomologoies-3}, suppose $\A$ is generic.  If $W\in\calL_2(\A)$ is a line, then (since $\A$ is generic) $W$ is contained in only two hyperplanes, say $H_1$ and $H_2$. By definition \cref{eq.V_Ar-1defn},
 we have
\begin{eqnarray*}
V_\A^{\br,1} &= &
\Big\{\sum_{H\supseteq W} f_He_H: W\in\calL_2(\A), \sum_{H\supseteq W} f_H{\ell}_H^{\br(H)+1}=0\Big\}, \text{\; hence} \\
& = &
\left\lbrace f_{H_1}e_{H_1}+f_{H_2}e_{H_2} : W\in\calL_2(\A) ~\mbox{and}~ H_1,H_2 \supseteq W, 
f_{H_1}{\ell}_{H_1}^{\br(H_1)+1}
+ f_{H_2}{\ell}_{H_2}^{\br(H_2)+1}=0\right\rbrace.
\end{eqnarray*}
Because $\ell_{H_1}$ and $\ell_{H_2}$ are linear forms,
$f_{H_1}{\ell}_{H_1}^{\br(H_1)+1}
+ f_{H_2}{\ell}_{H_2}^{\br(H_2)+1}=0$ if and only if 
$\ell_{H_2}^{\br(H_2)+1}\mid f_{H_1}$ and 
$\ell_{H_1}^{\br(H_1)+1}\mid f_{H_2}$. Using
this observation, it can be shown that as an $R$-module,
$V_\A^{\br,1}$ is generated by elements of
the form $\ell_{H_2}^{\br(H_2)+1}e_{H_1} - 
\ell_{H_1}^{\br(H_1)+1}e_{H_2}$.  But these elements
are also the generators of 
$\mbox{im}(\partial_2)$. Consequently, $\mbox{im}(\partial_2) =
V_{\A}^{\br,1}$, and so $\HH_1(\calJ^{\br})\cong \HK_1(\bm{\ell}^{\br})$.
\end{proof}

As a consequence of \Cref{thm.linkhomologoies} and \Cref{prop:EulerCharacteristicAndDimension}, we relate the dimension of spline spaces on the fan of a hyperplane arrangement to Koszul homology.

\begin{theorem}
\label{thm.SplineGenericArrangement}
Suppose $\A\subseteq\RR^3$ is an essential hyperplane arrangement and $\br:\A\to\ZZ_{\ge -1}$ is a smoothness distribution.  If $d\ge 0$, then
\[
\dim\calS^{\br}_d(\Sigma^\A)\le 2\dim R_d+ \sum_{\tau\in\Sigma^\A_2}\dim J(\tau)_d -\sum_{\gamma\in\Sigma_1^\A}\dim J(\gamma)_d+\dim \HK_1(\bm{\ell}^{\br})_d-\dim \HK_0(\bm{\ell}^{\br})_d.
\]
If $d\gg 0$, then
\[
\dim\calS^{\br}_d(\Sigma^\A)= 2\dim R_d+ \sum_{\tau\in\Sigma^\A_2}\dim J(\tau)_d -\sum_{\gamma\in\Sigma^\A_1}\dim J(\gamma)_d.
\]
If $\A$ is \textit{generic} and $d\ge 0$, then
\[
\dim\calS^{\br}_d(\Sigma^\A)=2\binom{d+2}{2}+2\sum_{\substack{H_1,H_2\in \A\\H_1\neq H_2}}\binom{d-\br(H_1)-\br(H_2)}{2}+\dim \HK_1(\bm{\ell}^{\br})_d-\dim \HK_0(\bm{\ell}^{\br})_d.
\]
\end{theorem}
\begin{proof}
The first inequality is immediate from \Cref{prop:EulerCharacteristicAndDimension}, \Cref{thm.linkhomologoies}\ref{thm.linkhomologoies-1},
and the fact that $R/J(\cv)=\HK_0(\bm{\ell}^{\br})$. The equality for $d\gg 0$ follows from \Cref{prop:EulerCharacteristicAndDimension} since $\HK_1(\ell^r)$ has finite length (\Cref{thm.linkhomologoies}\ref{thm.linkhomologoies-2}) and $R/J(\cv)$ has finite length.

For the final equality, suppose $\A$ is generic.  By \Cref{thm.linkhomologoies}\ref{thm.linkhomologoies-3}, $\HK_1(\bm{\ell}^{\br})\cong \HH_1(\calJ^\br_\bullet)$.  
Thus by \Cref{prop:EulerCharacteristicAndDimension}, the inequality
given in the previous paragraph is an equality in
the case that $\A$ is generic.   To complete the proof it
suffices to show that
\[
2\dim R_d+ \sum_{\tau\in\Sigma_2^\A}\dim J(\tau)_d -\sum_{\gamma\in\Sigma_1^\A}\dim J(\gamma)_d=2\binom{d+2}{2}+2\sum_{\substack{H_1,H_2\in \A\\H_1\neq H_2}}\binom{d-\br(H_1)-\br(H_2)}{2}.
\]
Clearly $2\dim R_d=2\binom{d+2}{2}$.  We show that 
\begin{equation}\label{eq:genericCancel}
\sum_{\tau\in\Sigma_2^\A}\dim J(\tau)_d -\sum_{\gamma\in\Sigma_1^\A}\dim J(\gamma)_d=2\sum_{\substack{H_1,H_2\in \A\\H_1\neq H_2}}\binom{d-\br(H_1)-\br(H_2)}{2}.
\end{equation}
To prove \cref{eq:genericCancel}, observe that if $\tau\in \Sigma^\A_2$ is contained in a hyperplane $H$, then $J(\tau)=\langle \ell_H^{\br(H)+1}\rangle$ and so 
\[
\dim J(\tau)_d=\binom{d+1-\br(H_1)}{2}.
\]
Since $\A$ is generic, every ray $\gamma\in \Sigma^\A_1$ is contained in the intersection of a unique pair of hyperplanes $H_1,H_2\in \A$.  It follows that $J(\gamma)=\langle \ell^{\br(H_1)+1}_{H_1},\ell^{\br(H_2)+1}_{H_2}\rangle$ and so
\[
\dim J(\gamma)_d = \binom{d+1-\br(H_1)}{2}+\binom{d+1-\br(H_2)}{2}-\binom{d-\br(H_1)-\br(H_2)}{2}.
\]
The term $\binom{d+1-\br(H)}{2}$ appears once in the formula for $\dim J(\tau)_d$ for every $\tau\in\Sigma^\A_2$ with $\tau\subseteq H$ and once in the formula for $\dim J(\gamma)_d$ for every $\gamma\in\Sigma^\A_1$ with $\gamma\subseteq H$.  Clearly there is a bijection between the number of two dimensional cones $\tau\in\Sigma^\A_2$ contained in $H$ and the number of rays $\gamma\in\Sigma^\A_2$ contained in $H$.  So the terms of the form $\binom{d+1-\br(H)}{2}$ all cancel on the left-hand side of \cref{eq:genericCancel}, leaving just the terms $\binom{d-\br(H_1)-\br(H_2)}{2}$ for each ray $\gamma\in\Sigma^\A_1$.  For each pair of hyperplanes $H_1,H_2\in\A$ there are two rays in $H_1\cap H_2$, leading to the coefficient of $2$ on the right hand side of \cref{eq:genericCancel}.  This completes the proof.
\end{proof}

\begin{remark}
\label{rem:HilbertPolynomial}
Recall from \Cref{ss:HilbertPoly}
that the {\it Hilbert polynomial} of an $R$-module 
$M$ is a polynomial $\HP_M(d)$ such that $\HP_M(d) = \dim M_d$
for all $d \gg 0$, that is, it agrees with the Hilbert function
of $M$ for almost all integers $d$.
For an essential arrangement $\A\subset\RR^3$, \Cref{thm.SplineGenericArrangement} shows that, for $d \gg 0$, the value of $\dim\calS^{\br}_d(\Sigma^\A)$ 
only relies on the first three terms since both $\HK_1(\bm{\ell}^{\br})$
and $\HK_0(\bm{\ell}^{\br})$ have finite length.  In particular, if $\A$ is generic,
we can determine the Hilbert polynomial of $\calS^\br(\Sigma)$
directly from the smoothness distribution of $\br$:
\[
\HP_{\calS^{\br}(\Sigma^\A)}(d) = (d+2)(d+1) 
+ \sum_{\substack{H_1,H_2\in \A\\H_1\neq H_2}}
(d- \br(H_1)-\br(H_2))(d-\br(H_1)-\br(H_2)-1).
\]
If $\A$ is not generic, it is still possible to determine the Hilbert polynomial of $\calS^\br(\Sigma)$ directly from $\br$ and the intersection lattice of $\A$ -- i.e. the containments between lines and planes.  The main technical hurdle is to compute the Hilbert polynomial of the ideal $J(\gamma)$, where $\gamma\in\Sigma^\A_1$.  This is done by Geramita-Schenck in~\cite{FatPoints}; we omit the details but see 
Example \ref{ex:A3HP} for an illustration of the computation.
\end{remark}

\section{Dimension computations in high degree via regularity}
\label{sec.regularity}

Suppose $\A$ is a hyperplane arrangement with at least three linearly independent hyperplanes.  If $d$ is sufficiently large, the dimension of the spline space $\calS^\br_d(\Sigma^\A)$ coincides with the \textit{Hilbert polynomial} $\HP_{\calS^\br(\Sigma^\A)}(d)$, as in \Cref{rem:HilbertPolynomial}.  According to \cref{thm.SplineGenericArrangement}, the Hilbert polynomial of $\calS^\br(\Sigma^\A)$ is determined by \textit{local data}, namely the Hilbert polynomials of the ideals $\{J(\tau):\tau\in\Sigma^\A_2\}$ and $\{J(\gamma):\gamma\in\Sigma^\A_1\}$.

The values of $d$ for which $\dim\calS^\br_d(\Sigma^\A)$ is given by $\HP_{\calS^\br(\Sigma^\A)}(d)$ can be bounded by the maximum degree in which the Koszul homologies $\HK_0(\bm{\ell}^\br)$ and $\HK_1(\bm{\ell}^\br)$ are non-zero.  This is determined by the (Castelnuovo-Mumford) regularity of the two modules.  Recall that $\HK_0(\bm{\ell}^\br)=R/I$, where $I=\langle \bm{\ell}^\br\rangle=\langle \ell_H^{\br(H)+1}:H\in\A\rangle$.  In this section, we first provide a bound on the regularity of $\HK_1(\bm{\ell}^\br)$ using the regularity of $I$, and then we provide several bounds on the regularity of $I$.  Putting these together, we derive bounds in \Cref{maintheoremUsingRegularity} on the values of $d$ for which $\dim\calS^\br_d(\Sigma^\A)$ is given by $\HP_{\calS^\br(\Sigma^\A)}(d)$.

If $M$ is a graded $R$-module, we use $\reg(M)$ to denote the regularity of $M$. We refer to standard texts (e.g. \cite{cohenMacaulay,CommAlg,Syz}) for the definition of regularity; for our purposes we need the following properties.

\begin{lemma} 
\label{lemma:propertiesOfRegularity}\cite[Section 4B]{Syz}
Let $R$ be a graded ring.
 \begin{enumerate}[label=({\arabic*})]
  \item If $M$ is an finite length $R$-module, then $\reg(M) = \max \{t \: : M_t\neq 0 \}$. 
  \item $\reg(A \oplus B) = \max(\reg(A),\reg(B))$.
  \item If $I \subset R$ is a homogeneous ideal, then $\reg(I) = \reg(R/I) + 1$.
\end{enumerate}
\end{lemma}

\subsection{Bounding the regularity of the first Koszul homology}

The following result allows us to compute the Koszul homology of the sequence $\boldsymbol{\ell}^{\br}$ in terms of the Koszul homology of a minimal generating set for the ideal generated by the sequence. 
It follows from a direct adaptation of the proof of \cite[Proposition 1.6.21]{cohenMacaulay} to a graded setting.

\begin{proposition}
    \label{prop: first Koszul homology nonminimal}
    Let $R$ be a graded ring and $\by=(y_1, \dots, y_n)$ a sequence in $R$ with $\deg(y_i)=r_i+1$ for all $i\in [n]$.  Suppose that $\boldsymbol{y}'=
    (
    y_1,\ldots, y_p
    )$
    where $\langle \boldsymbol{y}\rangle=\langle \boldsymbol{y'}\rangle$ and $p\le n$.
    Then there is an isomorphism of graded modules 
    \begin{equation}
        \label{eq: nonminimal first Koszul homology} 
        \HK_1(\boldsymbol{y})
        \cong
        \HK_1(\boldsymbol{y}')
        \oplus
        \bigoplus_{i = p+1}^n 
        \frac{R}{\langle \boldsymbol{y}' \rangle}
        (-r_i - 1). 
    \end{equation}
\end{proposition}

Formulas for the regularity of Koszul homologies of ideals were studied in \cite{Bruns-Conca-Romer-2011b, Bruns-Conca-Romer-2011,Conca-Murai-2015}.  We use \cref{prop: first Koszul homology nonminimal} and \cite[Proposition~3.3]{Bruns-Conca-Romer-2011b} to deduce a bound on the regularity of $\HK_1(\by)$.

\begin{proposition}
\label{regularityBoundH1}
Let $R$ be a polynomial ring with standard grading.  Let $\by=(y_1, \dots, y_n)$ be a sequence in $R$ with $\deg(y_i)=r_i+1$ for all $i\in [n]$, set $I=\langle \by\rangle$, and assume that $R/I$ has finite length.  Suppose that $\boldsymbol{y}'=
    (
    y_1,\ldots, y_p
    )$
    where $\langle \boldsymbol{y}\rangle$ is minimally generated by $\boldsymbol{y'}$.

If $n=p$ then $\reg(\HK_1(\by))\le 2\reg(I)$.  If $n>p$, assume that $r_{p+1}\le \cdots \le r_n$.  Then,
$$\reg(\HK_1(\by)) \leq \reg(I) + \max \left \{\reg(I), r_n\right \}.$$
\end{proposition}

\begin{proof}
Since $R/I$ has finite length, \cite[Proposition~3.3]{Bruns-Conca-Romer-2011b} gives $\reg(\HK_1(\by'))\le 2\reg(I)$.  If $n=p$, we are done.  Otherwise, using also \Cref{prop: first Koszul homology nonminimal} and \Cref{lemma:propertiesOfRegularity}, we have
\begin{align*}
\reg(\HK_1(\by)) = & \reg\bigg(\HK_1(\by') \oplus \bigoplus_{i = p+1}^n(R/I)(-r_i - 1)\bigg)= \max \{\reg(\HK_1(\by')),\reg(R/I) +r_n + 1)\} \\ = & \reg(I) + \max \{\reg(\HK_1(\by')) - \reg(I),r_n\}
\le \reg(I)+\max \{\reg(I),r_n\}. 
\end{align*}
\end{proof}

As the next result shows, we can bound when
the Hilbert function and the Hilbert polynomial of 
$\calS^{\br}(\Sigma^A)$ agree in terms of the regularity
of the ideal  
$I=\langle \bm{\ell}^\br\rangle=\langle \ell_H^{\br(H)+1}:H\in\A\rangle$.
The following theorem can be deduced from the regularity bound in \cref{regularityBoundH1} and the dimension formula in \cref{thm.SplineGenericArrangement}.
\begin{theorem}
\label{maintheoremUsingRegularity1}
Suppose $\A=\{H_1,\ldots,H_n\}$ is an essential hyperplane arrangement in $\RR^3$, $\br\colon\A\to\ZZ_{\ge -1}$ is a smoothness distribution, and put $\br(H_i)=r_i$.  Assume that the ideal $I = \langle \bm{\ell}^{\br}\rangle$ is minimally generated by $\ell_1^{r_1 + 1},\dots,\ell_p^{r_p + 1}$.  Suppose that either $n=p$ and $d>2\reg(I)$ or $n>p$, $r_{p+1}\le \cdots \le r_n$, and $d > \reg(I) + \max \{\reg(I),r_n\}$.
Then
\[
\dim\calS^{\br}_d(\Sigma^\A)= 2\dim R_d+ \sum_{\tau\in\Sigma^\A_2}\dim J(\tau)_d -\sum_{\gamma\in\Sigma^\A_1}\dim J(\gamma)_d.
\]
If $\A$ is \textit{generic}, then 
\[
\dim\calS^{\br}_d(\Sigma^\A)=2\binom{d+2}{2}+2\sum_{\substack{H_1,H_2\in \A\\H_1\neq H_2}}\binom{d-\br(H_1)-\br(H_2)}{2}.
\]
\end{theorem}

\subsection{When does the dimension of the spline space become polynomial?}

As just shown in  \Cref{maintheoremUsingRegularity1},
we can bound the degrees $d$ for which $\dim\calS^\br_d(\Sigma^\A)$ is given by its Hilbert polynomial $\HP_{\calS^\br(\Sigma^\A)}(d)$ in
terms of the regularity of the ideal
$I=\langle \bm{\ell}^\br\rangle=\langle \ell_H^{\br(H)+1}:H\in\A\rangle$.  We further
refine this result in this section by showing how
 either the smoothness function ${\bf r}$ or 
geometric information related to $I$ can  give bounds
on ${\rm reg}(I)$.

Throughout the rest of the section (and also in \Cref{sec:smallnumber}), we use that the dual point of the linear form $\ell_i = \alpha_0\, x + \alpha_1\, y + \alpha_2\, z$ is $P_i = [\alpha_0:\alpha_1:\alpha_2] \in \PP^2 = \PP(\R^3)$.

\begin{proposition}\label{prop:regularityBoundsPowerIdeals}
Suppose $\A$ is an essential hyperplane arrangement in $\RR^3$ of $n$ hyperplanes and $\br\colon\A\to\ZZ_{\ge -1}$ is a smoothness distribution. Let $I = \langle \bm{\ell}^{\br}\rangle$.
 \begin{enumerate}[label=({\arabic*})]
\item Let $i,j,k$ be indices so that $\ell_i,\ell_j,\ell_k$ are linearly independent.  Then 
$
\reg(R/I)\le r_i+r_j+r_k.
$
\item Suppose there are indices $1\le i\le j \le s\le t\le n$ so that every proper subset of $\{\ell_i,\ell_j,\ell_s,\ell_t\}$ is linearly independent.  Then
$
\reg(R/I)\le \left\lfloor\dfrac{r_i+r_j+r_s+r_t}{2}\right\rfloor.
$
\item Suppose $\br$ is constant, so $r_i=r$ for all $1\le i\le n$ for some integer $r$.  Then
$
\reg(R/I)\le r\left(\dfrac{\alpha+1}{\alpha-1}\right),
$
where $\alpha$ is the smallest degree of a non-zero homogeneous polynomial vanishing on the dual set of points $P_1,\ldots,P_n$.
\end{enumerate}
\end{proposition}
\begin{proof}
For the first two inequalities, we use that $R/I$ has finite length (since $\A$ is essential), so if $J\subseteq I$ is an ideal and $R/J$ has finite length, then $\reg(R/I)\le \reg(R/J)$.

For the first inequality, $J_1=\langle \ell_i^{r_i+1},\ell_j^{r_j+1},\ell_k^{r_k+1}\rangle\subseteq I$ and $J_1$ is a complete intersection, so $\reg(R/J_1)=r_i+r_j+r_k$.

For the second inequality, $J_2=\langle \ell_i^{r_i+1},\ell_j^{r_j+1},\ell_s^{r_s+1}, \ell_t^{r_t+1}\rangle\subseteq I$ and, by the hypotheses, there is a change of coordinates so that $J_2=\langle x^{r_i+1},y^{r_j+1},z^{r_s+1},(x+y+z)^{r_t+1}\rangle$.  With $J_1=\langle x^{r_i+1},y^{r_j+1},z^{r_s+1}\rangle$, a result of Stanley~\cite{Stanley-1980} shows that $R/J_1$ has the strong Lefschetz property and $x+y+z$ is a strong Lefschetz element. Consequently, the computation of \cite[Lemma 2.5]{MiglioreMiro-Roig2002} applies, yielding
$
\reg(R/J_2)=\left\lfloor(r_i+r_j+r_s+r_t)/2\right\rfloor.
$
For the third inequality, let $X=\{P_1,\ldots,P_n\}$ and $I_X$ be the ideal of homogeneous polynomials vanishing on $X$.  It follows from \cite[Proposition~3.11]{Michael-Nelly-SIAGA} (see also \cite[Corollary~4.22 and Example~4.23]{DNS23}) that $\reg(S/I)\le \frac{\widehat{\alpha}(I_X)r}{\widehat{\alpha}(I_X)-1}$, where $\widehat{\alpha}(I_X)$ is the so-called \textit{Waldschmidt constant} of $I_X$.  Since Chudnovsky's conjecture is known for points in $\mathbb{P}^2$ (see e.g. \cite[Proposition~3.1]{HH13}), it follows that $\widehat{\alpha}(I_X)\ge (\alpha(I_X)+1)/2$, where $\alpha(I_X)=\alpha$ is the smallest degree of a non-zero homogeneous polynomial which vanishes on $X$.  Thus
$
\reg(S/I)\le \dfrac{\widehat{\alpha}(I_X)r}{\widehat{\alpha}(I_X)-1}\le \dfrac{\alpha+1}{\alpha-1}r.
$
\end{proof}

We now come to the main result of this section, which refines \Cref{thm.SplineGenericArrangement} by bounding how large $d$ must be for $\dim \calS^{\br}_{d}(\Sigma^\A)$ to agree with $\HP_{\calS^{\br}(\Sigma^\A)}(d)$, using either
the smoothness distribution or the geometry
of the dual points.

\begin{theorem}
\label{maintheoremUsingRegularity}
Suppose $\A$ is an essential hyperplane arrangement in $\RR^3$ of $n$ hyperplanes and $\br\colon\A\to\ZZ_{\ge -1}$ is a smoothness distribution. Assume that the ideal $I = \langle \bm{\ell}^{\br}\rangle$ is minimally generated by $\ell_1^{r_1 + 1},\dots,\ell_p^{r_p + 1}$.  We further assume that $r_1\le \cdots \le r_p$ and $r_{p+1}\le \cdots \le r_n$.  Suppose that one of the following conditions holds:
 \begin{enumerate}[label=({\arabic*})]
\item $d> r_1+r_2+r_3+1+\max\{r_1+r_2+r_3+1,r_n\}$.
\item There are indices $i,j,s,t$ so that every subset of size three of $\{\ell_i,\ell_j,\ell_s,\ell_t\}$ is linearly independent and $d>\lfloor(r_i+r_j+r_s+r_t)/2\rfloor+1+\max\{\lfloor(r_i+r_j+r_s+r_t)/2\rfloor+1,r_n\}$.
\item $\br$ is constant with $r_i=r$ for $i=1,\ldots,n$ for some non-negative integer $r$ and $d>2r\frac{\alpha+1}{\alpha-1}+2$.
\end{enumerate}

Then
\[
\dim\calS^{\br}_d(\Sigma^\A)= 2\dim R_d+ \sum_{\tau\in\Sigma^\A_2}\dim J(\tau)_d -\sum_{\gamma\in\Sigma^\A_1}\dim J(\gamma)_d.
\]
If $\A$ is \textit{generic}, then 
\[
\dim\calS^{\br}_d(\Sigma^\A)=2\binom{d+2}{2}+2\sum_{\substack{H_1,H_2\in \A\\H_1\neq H_2}}\binom{d-\br(H_1)-\br(H_2)}{2}.
\]

\end{theorem}

\begin{proof}
Apply \Cref{thm.SplineGenericArrangement}, \Cref{regularityBoundH1}, and \Cref{prop:regularityBoundsPowerIdeals}, keeping in mind that $\reg(I)=\reg(S/I)+1$.
\end{proof}

\begin{remark}
The condition that there are four linear forms $\ell_i,\ell_j,\ell_s,\ell_t$ so that every subset of size three of $\{\ell_i,\ell_j,\ell_s,\ell_t\}$ is linearly independent is equivalent to saying that the hyperplane arrangement $\A$ is \textit{indecomposable}.  That is $\A$ is not a product of a one-dimensional and a two-dimensional hyperplane arrangement.
\end{remark}

\begin{remark}
Suppose $\br$ is constant with $r_i=r$ for every $i=1,\ldots,n$.  Then the bound in \Cref{maintheoremUsingRegularity} part $(2)$ is $d>4r+2$.  The bound in part $(3)$ performs at least as well as the bound in part $(2)$ as long as $\alpha\ge 3$.  If $\alpha>3$ (that is, the dual points do not lie on a cubic) then the bound in part $(3)$ is strictly better.
\end{remark}

\begin{remark}
Let $d^*_{\A,\br}$ be the maximum value at which the Hilbert function of $\calS^\br(\Sigma^\A)$ is not equal to the Hilbert polynomial of $\calS^\br(\Sigma^\A)$.  This is sometimes called the \textit{postulation number}.  The three parts of \Cref{maintheoremUsingRegularity} give three bounds on $d^*_{\A,\br}$.  These bounds are best when $\A$ is a generic arrangement, but even in the generic setting the bounds are not always optimal.  For instance, if $\A$ is generic and $\br=\bm{0}$ assigns the constant $0$ to every hyperplane, then \Cref{thm:C0} implies that $d^*_{\A,\bm{0}}=1$.  However, the bounds $(1),(2),$ and $(3)$ in \Cref{maintheoremUsingRegularity} only give $d^*_{\A,\bm{0}}\le 2$.  If $\A$ is generic with four hyperplanes and $\br$ assigns the constant $r$ to every hyperplane, then $(2)$ and $(3)$ both yield $d^*_{\A,\bm{r}}\le 4r+2$.  However, \Cref{thm:threefourHyperplanes} implies that $d^*_{\A,\br}=4r+1$.  (The discrepancy appears to be caused by the fact that the bound $\reg(\HK_1(I))\le 2\reg(I)$ is off by exactly one.)  On the other hand, the bound in (3) tends to be tight if $\A$ has many hyperplanes and is very generic.  According to \Cref{theorem: spline dimension many planes}, if $\A$ is generic, $\br$ assigns the constant $r$ to every hyperplane, and $\langle\bm{\ell}^{\bm{r}}\rangle=\langle x,y,z\rangle^{r+1}$, then $d^*_{\A,\br}=2r+2$.  This agrees with the bound in $(3)$ if $\alpha>4r+1$.  However, the geometry of $\A$ can have an impact here.  For instance, if $\br=\bm{1}$ and the dual points of $\A$ all lie on a conic, \Cref{thm:C1} implies that $d^*_{\A,\bm{1}}=4$.  However, $(3)$ only implies that $d^*_{\A,\bm{1}}\le 8$. 
\end{remark}

\section{Complete dimension computation for small generic hyperplane arrangements}\label{sec:smallnumber}
In this section, we use \Cref{thm.SplineGenericArrangement} to provide a formula for the dimension of $\calS^{\br}_d(\Sigma^\A)$ (for all $d\ge 0$) by establishing the Hilbert function of the Koszul homology modules $\HK_1(\bm{\ell}^{\br})$ when $\A$ is generic and has five or less hyperplanes.

In the results in this section, a combinatorial property of Koszul homology modules is also established: for generic hyperplane arrangements with five or less hyperplanes, the dimension of $\HK_1(\bm{\ell}^{\br})$ only depends on the smoothness distribution.  We say that the obtained formula for $\dim\calS^\br_d(\Sigma^\A)$ is \textit{combinatorial} because it does not depend on the exact linear forms defining $\A$, but only on the fact that $\A$ is generic.  This is no longer true for six hyperplanes: we provide in \Cref{ex:sixhyperplanes} two generic hyperplane arrangements of six hyperplanes whose corresponding spline modules have different Hilbert functions (for constant smoothness $r=2$).

The arguments in this section require properties of Koszul homology modules that go beyond the foundations established so far; we record them next. Recall that the grade of an ideal is the length of the longest regular sequence contained in the ideal.

\begin{proposition}
\label{prop:someKoszul}
Let $\kk$ be a field, $R=\kk[x_1, \dots, x_m]$ and $\by=(y_1, \dots, y_n)$  in $R$ with $\deg(y_i)=r_i+1$ for all $i\in [n]$. Set $\overline r=\sum_{i=1}^n\deg(y_i)=n+\sum_{i=1}^nr_i$. If $\grade (\langle \by\rangle)=m$, then 
 \begin{enumerate}[label=({\arabic*})]
\item $\HK_i(\by)_{j}\cong \HK_{n-m-i}(\by)_{\overline{r}-m-j}$ for all $i,j\ge 0$. 
\smallskip
\item Assume $m=3$. For each $d\ge 0$, 
\[
\dim \HK_1(\by)_d-\dim \HK_0(\by)_d=\sum_{i=2}^{n-3} (-1)^i\dim \HK_i(\by)_d-\!\sum_{A \subseteq [n]} (-1)^{|A|} \binom{d+2-|A|-\!\sum_{i\in A}r_i}{2}.
\]
\end{enumerate}
\end{proposition}

\begin{proof}
Set $I=\langle \by \rangle$. 
To prove Statement (1), we use \cite[Lemma 5.7]{chardin2004regularity}, that gives 
\begin{equation}
\label{eq:dual}
\HK_i(\by)^{\vee\vee}(\overline{r}+a_S)\cong \HK_{n-m-i}(\by)^\vee\quad\text{for all $i$,}
\end{equation}
where $a_S=-m$ is the $a$-invariant of $R$ and   $(-)^\vee=\Ext_R^m(-,\omega_R)$ with $\omega_R$ the canonical module of $R$.  For any finite length $R$-module $U$, the module $U^\vee $ is the Matlis dual of $U$, where $U^\vee\cong\Hom_{\kk}(U,\kk)$ and $(U^\vee)_j\cong \Hom_{\kk}(U_{-j},\kk)$ for all $j$; see \cite{cohenMacaulay} for basics of Matlis duality. Further, since $\grade(I)=m$ and $I$ annihilates Koszul homology, the modules $\HK_i(\by)$ are finite length for all $i\ge 0$. 
Consequently, \cref{eq:dual} gives an isomorphism of graded $R$-modules 
$
\HK_i(\by)(\overline{r}-m)\cong \HK_{n-m-i}(\by)^\vee$
that translates to an isomorphism of $\kk$-vector spaces 
$\HK_i(\by)_{j+\overline{r}-m}\cong \HK_{n-m-i}(\by)_{-j} \quad\text{for all $i,j$}.$
After a re-indexing,  this yields Statement (1). 

For Statement (2), we use the fact that in a finite complex of finitely generated graded $R$-modules, the alternating sum of the Hilbert functions of the homologies is equal to the alternating sum of the Hilbert functions of the original modules in the complex. Applying this to the Koszul complex $\K_\bullet(\by)$  we obtain
 \begin{equation}
 \label{e:alternate}
 \dim \HK_1(\by)_d-\dim \HK_0(\by)_d=\sum_{i=2}^n (-1)^i\dim \HK_i(\by)_d-\sum_{i=0}^n (-1)^i\dim\K_i(\by)_d\,.
 \end{equation}
The grade sensitivity property of Koszul homology, see \cite[Theorem 1.6.17(b)]{cohenMacaulay}, gives $\HK_i(\by)=0$ for $i>n-3$. 
The formula follows since $m=3$ and $\K_i(\by)\cong \bigoplus\limits_{A\subseteq [n],|A|=i} R(-\sum\limits_{i\in A}(r_i+1))$.
\end{proof}

\subsection{Powers of linear forms and fat points} 
The following result elegantly transforms the problem of studying ideals generated by powers of linear forms into the geometric problem of studying ``fat points", i.e., finding the number of degree $d$ curves in $\mathbb{P}^2$ that pass through a specified set of points with prescribed multiplicities. For a linear form $\ell = \alpha x + \beta y + \gamma z \in R$, we define its dual point as 
\begin{equation}
\label{eq:dualPoint}
P = [\alpha : \beta : \gamma] \in \mathbb{P}^2= \PP(\RR^3).
\end{equation}
The defining ideal of this point is given by $\mathfrak{p} = \langle \alpha Y -\beta X, \beta Z - \gamma Y, \gamma X - \alpha Z \rangle \subset S$.

\begin{theorem}[Theorem IIA \cite{Iarrobino}]
\label{thm:Apolarity}
Let $I = \langle \ell_1^{r_1+1}, \dots, \ell_n^{r_n+1} \rangle \subset R$ be an ideal generated by powers of linear forms, and let $\mathfrak{p}_i \subset S$ be the defining ideal of the point $P_i \in \mathbb{P}^2$ dual to $\ell_i$. 
With the convention that $\mathfrak{p}^m=S$ when $m\le 0$,  we have
$$\dim (R/I)_d = \dim \Big( \bigcap_{i=1}^n \mathfrak{p}_i^{d - r_i} \Big)_d \quad\text{for all $d\ge 0$}.$$
\end{theorem}

\Cref{thm:Apolarity} is a consequence of the \textit{apolarity} action; see the excellent survey~\cite{Geramita}.

\subsection{The case of three, four and five hyperplanes}

We provide formulas for $\dim\calS^{\br}_d(\Sigma^\A)$ when $d = 3,4,5$. In all of these cases, the fomulas are combinatorial.

\begin{remark}
\label{r:generic-grade}
If $\A\subseteq \R^3$ is a generic hyperplane arrangement with $n$ hyperplanes and  $I=\langle\bm{\ell}^{\br}\rangle$, then $\grade(I)=\min\{n,3\}$. Indeed, up to a change of variables, we may assume that 
$\{x_1, \dots, x_g\}\subseteq \{\ell_H\colon H\in \A\}$, where $g=\min\{3,n\}$.
\end{remark}

In what follows, we  set $\br(\emptyset)=0$ and
$
\br(A)=\sum_{H\in A}\br(H)$ for $A\subseteq \A$. 
\begin{theorem}[{\bf Three and four hyperplanes}]
\label{thm:threefourHyperplanes}
Suppose $\A\subseteq\RR^3$ is a generic hyperplane arrangement with $n\in\{3,4\}$ hyperplanes and $\br:\A\to\ZZ_{\ge -1}$ is a smoothness distribution which is constant along hyperplanes.  Then for all $d\ge 0$
\[
\dim\calS^{\br}_d(\Sigma^\A)=2\binom{d+2}{2}+2\sum_{A\subseteq \A, |A|=2}\binom{d-\br(A)}{2}-\sum_{A\subseteq \A}(-1)^{|A|}\cdot \binom{d+2-|A|-\br(A)}{2}.
\]
In particular, $\dim\calS^{\br}_d(\Sigma^\A)$ is a combinatorial invariant. 
\end{theorem}
\begin{proof}
Since $\grade(\bm{\ell}^\br)=3$ by \cref{r:generic-grade}, the statement follows directly from \cref{thm.SplineGenericArrangement} and \cref{prop:someKoszul}(3). 
\end{proof}

\begin{theorem}[{\bf Five hyperplanes}]
\label{thm:fiveHyperplanes}
Suppose $\A\subseteq\RR^3$ is a generic hyperplane arrangement with $n=5$ hyperplanes and $\br:\A\to\ZZ_{\ge -1}$ is a smoothness distribution which is constant along hyperplanes.  Then for all $d\ge 0$
\begin{align*}
\dim\calS^{\br}_d(\Sigma^\A)=&2\binom{d+2}{2}+2\sum_{A\subseteq \A, |A|=2}\binom{d-\br(A)}{2}-\sum_{A\subseteq \A}(-1)^{|A|}\cdot \binom{d+2-|A|-\br(A)}{2}\\
&+\dim (R/\langle \bm{\ell}^{\br}\rangle)_{\br(\A)+|\A|-3-d}.
\end{align*}
Furthermore, there is an algorithm to compute $\dim(R/\langle \bm{\ell}^{\br}\rangle)_d$ for any $d$,  which only depends on $\br$.  In particular, $\dim\calS^{\br}_d(\Sigma^\A)$ is a combinatorial invariant. 

\end{theorem}
\begin{proof}
We have $\grade(\langle \bm{\ell}^{\br}\rangle)=3$ by \cref{r:generic-grade}. Using \cref{prop:someKoszul}(1) with $m=3$ and $n=5$, note that
\[
\HK_2(\bm{\ell}^\br)_d\cong \HK_0(\bm{\ell}^\br)_{\br(\A)+|\A|-3-d}\cong (R/\langle \bm{\ell}^\br\rangle)_{\br(\A)+|\A|-3-d}.
\]
Using this computation, the formula for $\dim\calS^{\br}_d(\Sigma^\A)$ is a direct consequence of 
\cref{thm.SplineGenericArrangement} and \cref{prop:someKoszul}(3). 
Since we assume that the hyperplane arrangement is generic, no three of the points $P_i$ are aligned. In particular, the dual point $P_1,\ldots,P_5$ lie on a nonsingular (irreducible) conic.  An algorithm due to Catalisano~\cite[Theorem~3.1]{Catalisano-1991} can be used to compute the Hilbert function of any ideal of fat points for which the underlying points lie on an irreducible conic.  As noted in \cite[Remark 3.2]{Catalisano-1991},
the Hilbert function only depends on the multiplicities $m_i$ and not on the relative position of the points on the conic 
(see \Cref{thm:catalisanoAlgorithm} for details).  Thus, the Hilbert function of $\langle \bm{\ell}^{\br} \rangle$ is independent of $\A$, and only depends on the multiplicities of the points determined by $\br$.
\end{proof}

In \cref{ex:5forms2smoothness}, we illustrate \Cref{thm:fiveHyperplanes} and Catalisano's algorithm by computing $\dim\calS^{2}_d(\Sigma^\A)$ for five generic hyperplanes.

\subsection{The case of six hyperplanes}
We provide an example in which we show that in the case of six hyperplanes, the formula for $\dim\calS^\br_d(\Sigma^\A)$ is not combinatorial. Namely, there exist two sequences of six hyperplanes and a  smoothness distribution $\br: \A \xrightarrow[]{} \mathbb{Z}_{\geq -1}$ such that the corresponding Koszul homology modules $\HK_1(\bm{\ell}^{\br})$ have different Hilbert series.

\begin{example}\label{ex:sixhyperplanes}
    Let $\phi \in \RR$ be a solution of the equation $x^2-x-1=0$. Consider the six linear forms $\tilde{\ell}_1=x, \tilde{\ell}_2=y,\tilde{\ell}_3=z,\tilde{\ell}_4=x+y+z$, $\tilde{\ell}_5=x-\phi y+(\phi+1)z$, and $\tilde{\ell}_6=x+(\phi+1)y-\phi z$ and let $\br$ be the constant smoothness distribution $(2,2,2,2,2,2)$. Computing with \texttt{Macaulay2}, the Hilbert series of $\HK_1({\tilde{\bm{\ell}}^{\br}})$ equals 
$
\hilb_{\HK_1({\tilde{\bm{\ell}}^{\bm{2}}})}(t) = 3t^4+15t^5+17t^6+10t^7+t^8.
$
However, for the sequence $\ell_1 = x,\ell_2 = y, \ell_3 = z, \ell_4 = x + y + z,\ell_5 = x + 2y + 2z, \ell_6 = 2x + y + 2z$, the Hilbert series of $\HK_1(\bm{\ell}^{\br})$ equals
$
\hilb_{\HK_1({\bm{\ell}^{\bm{2}}})}(t) = 3t^4+15t^5+17t^6+9t^7.
$
Both sequences define a generic arrangement, meaning that no three of the hyperplanes intersect on a line. Therefore, by \Cref{thm.SplineGenericArrangement}, we deduce that
\begin{align*}
    \dim\calS^2_7(\Sigma^{\tilde{\A}}) = & 2\binom{9}{2} + 2\binom{6}{2}\binom{3}{2} + \dim \HK_1(\tilde{\bm{\ell}}^{\bm{2}}) = 72 +90+ 10 = 172 \mbox{ and} \\ 
    \dim\calS^2_7(\Sigma^{\A}) = & 2\binom{9}{2} + 2\binom{6}{2}\binom{3}{2} + \dim \HK_1(\bm{\ell}^{\bm{2}}) = 72 +90+ 9 = 171,
\end{align*}
where $\tilde{\A}$ and $\A$ are the hyperplane arrangements 
corresponding to the sequences $\tilde{\bm{\ell}}$ and $\bm{\ell}$.

We also observe that a distinguishing feature between the two sequences is that while for $\tilde{\bm{\ell}}^{\bm{2}}$ there exists a linear sygygy in the module of second Koszul differentials, this does not happen for the sequence $\bm{\ell}^{\bm{2}}$. This seems to be a tip of the iceberg situation, as we have observed a similar behaviour for higher constant smoothness distributions.
\end{example}

\begin{remark}
If $\A$ is a generic arrangement with seven or more hyperplanes, the dimension of $\calS^1_4(\Sigma^\A)$ is not combinatorial -- see \Cref{thm:C1}.  The geometry in the case of \Cref{thm:C1} is clear; $\dim \calS^1_4(\Sigma^\A)$ has one value when the dual points all lie on a conic and a different value otherwise.  On the other hand, we have no explanation for the geometry that contributes to the different dimensions of $\calS^2_7(\Sigma^\A)$ arising in \Cref{ex:sixhyperplanes}.  This suggests that describing $\dim \calS^r_{3r+1}(\Sigma^\A)$ could be a fundamental challenge when $\A$ is a generic arrangement.
\end{remark}

\section{The generic case with constant smoothness and many hyperplanes}
\label{sec.uniformsmoothness}

In this section, we compute the exact  dimension of the spline space $\calS^\br_d(\Sigma)$ under two assumptions: 
 \begin{enumerate}[label=({\arabic*})]
    \item The smoothness distribution $\boldsymbol{r}:\mathscr{A} \xrightarrow{} \Z_{\geq -1}$ is constant, that is, $\br(H_i) = r$ for $i\in [n]$. 
    \item The ideal $\langle \bm{\ell}^{\br}\rangle$ is equal to the ideal $\langle x,y,z\rangle^{r+1}$.
\end{enumerate}
In particular, (1) implies that all elements of the sequence $\boldsymbol{\ell}^{\br}$ have the same degree $r+1$.  For condition (2) to hold, there must be a subset of $\{\ell_i^{r+1}:1\le i\le n\}$ with cardinality $\binom{r+3}{2}$ which is linearly independent.  Equivalently, via apolarity, there is no homogeneous polynomial of degree $r+1$ which vanishes on the dual points $P_1,\ldots,P_n$.  This is true for `most' collections of at least $\binom{r+3}{2}$ points in $\mathbb{P}^2(\RR)$.

Since $\langle \bm{\ell}^{\br}\rangle=\langle x,y,z\rangle^{r+1}$, the $0$'th Koszul homology of the sequence satisfies $\dim \HK_0(\boldsymbol{\ell}^{\boldsymbol{r}})_d
=\binom{d+2}{2}$ if  $d \le r$ and $\dim \HK_0(\boldsymbol{\ell}^{\boldsymbol{r}})_d=0$ if $d>r$. In order to provide an analogous description of the first Koszul homology of $\bm{\ell}^{\br}$, we first describe the Hilbert function of the first Koszul homology of $\langle x,y,z\rangle^{r+1}$.
\begin{proposition}
\label{p:Koszul-homology-powers}
The first Koszul homology of $\langle x,y,z\rangle^{r+1}\subseteq \RR[x,y,z]$ has Hilbert function
\[
\dim \HK_1(\langle x,y,z\rangle^{r+1})_d=
\begin{cases}
0 & \text{if } d\le r \text { or } d\ge 2r+3\\
\binom{r+3}{2}\binom{d+1-r}{2}-\binom{d+2}{2} & \text{if } r+1\le d\le 2r+1\\
\binom{r+3}{2}\binom{d+1-r}{2}-\binom{d+2}{2}-\binom{\binom{r+3}{2}}{2} & \text{if } d=2r+2.
\end{cases}
\]
\end{proposition}
\begin{proof}
Set  $R=\RR[x,y,z]$. Consider the short exact sequence
\[
0\to \syz(\langle x,y,z\rangle^{r+1})\to R(-r-1)^{\binom{r+3}{2}}\to \langle x,y,z\rangle^{r+1}\to 0.
\]
From this we deduce that
$
\dim \syz(\langle x,y,z\rangle^{r+1})_d=
\begin{cases}
0 & \text{if } d\le r\\
\binom{r+3}{2}\binom{d+1-r}{2}-\binom{d+2}{2} & \text{if } d\ge r+1
\end{cases}.
$

By \Cref{sec:KoszulConnection} (see \cref{eq.koszulkernel}), $\HK_1(\langle x,y,z\rangle^{r+1})$ is the quotient of $\syz(\langle x,y,z\rangle^{r+1})$ by the Koszul syzygies of the minimal monomial generators of $\langle x,y,z\rangle^{r+1}$, and the Koszul syzygies all have degree $2r+2$.  Thus the formula follows for $d\le 2r+1$. 

In degree $d=2r+2$, we claim the Koszul syzygies are linearly independent.  Since there are $\binom{\binom{r+3}{2}}{2}$ of them, the formula follows for $d=2r+2$ once we establish this claim.  Let $\lambda=\binom{r+3}{2}$ and order the degree $r+1$ monomials of $\R[x,y,z]$ lexicographically as  $m_1=x^{r+1},\ldots,m_\lambda=z^{r+1}$.  The Koszul syzygies are the syzygies $K_{ij}=m_jb_i-m_ib_j, 1\le i\le j\le \lambda$ in the free module $\bigoplus_{i=1}^\lambda R(-r-1)b_{i}$, where $b_{i}$ is a standard basis symbol corresponding to the monomial $m_i$.  Suppose to the contrary that there is a non-trivial linear dependency $\sum_{1\le i<j\le \lambda} c_{ij}K_{ij}=0$ where $c_{ij}\in \R$ for all $1\le i<j\le \lambda$.  Since the linear dependency is non-trivial, there are $s,t$ satisfying $1\le s<t\le \lambda$ so that $c_{st}\neq 0$.  Then $\sum_{i=1}^\lambda c_{si}m_ib_s=0$, implying that there is a linear dependency among the monomials of degree $r+1$.  This is a contradiction since the monomials of degree $r+1$ are linearly independent.  Thus the Koszul syzygies are linearly independent and the formula follows for $d=2r+2$.
The fact that $\dim \HK_1(\langle x,y,z\rangle^{r+1})_d=0$ for $d\ge 2r+3$ follows from \cite[Corollary~3.7]{Bruns-Conca-Romer-2011}.
\end{proof}

Now we describe the Hilbert function of the first Koszul homology of the sequence $\boldsymbol{\ell}^{\br}$.

\begin{corollary}
    \label{corollary: nonminimal first koszul}
    Let $n \geq p = \binom{r+3}{2}$    and $
    \boldsymbol{\ell}^{\boldsymbol{r}}
    = 
    (
    \ell_1^{r + 1},\ldots, \ell_{p}^{r + 1} , \ell_{p+1}^{r + 1} , \ldots , \ell_n^{r+ 1} 
    )$, where the first $p$ linear forms satisfy
    $
    \langle 
    \ell_1^{r+1}
    ,
    \ldots
    , 
    \ell_p^{r+1}
    \rangle 
    = 
    \langle 
    x,y,z
    \rangle^{r+1}
    $.
    We have a graded isomorphism  
    \begin{equation}
    \label{eq: nonminimal first Koszul homology}
    \HK_1(\boldsymbol{\ell}^{\boldsymbol{r}})
        \cong 
        \HK_1(\langle x,y,z \rangle^{r+1})
        \oplus
        \bigoplus_{i=p+1}^n
        \dfrac{R}{\langle x,y,z\rangle^{r+1}}(-r-1)
        \ .
    \end{equation}
    In particular, 
$
\dim \HK_1( \boldsymbol{\ell}^r )_d=
\begin{cases}
0 & \text{if } d\le r \text{ or } d\ge 2r+3\\
n \binom{d + 1 - r}{2} 
- \binom{d + 2}{2}  & \text{if } r+1\le d\le 2r+1\\
\binom{r+3}{2}\binom{d+1-r}{2}-\binom{d+2}{2}-\binom{\binom{r+3}{2}}{2} & \text{if } d=2r+2.
\end{cases}
$
\end{corollary}

\begin{proof}
Since $I = \langle \ell_1^{r + 1},\ldots, \ell_{p}^{r + 1} , \ell_{p+1}^{r + 1} , \ldots , \ell_n^{r+ 1} \rangle = \langle x,y,z \rangle^{r+1}$, 
the graded isomorphism \cref{eq: nonminimal first Koszul homology} follows from \cref{prop: first Koszul homology nonminimal}.  The dimension formula now follows by applying \Cref{p:Koszul-homology-powers} and simplifying.
\end{proof}

We can now combine the above results to derive a formula
for the dimensions of the spline space under some mild additional hypotheses.
\begin{theorem}
\label{theorem: spline dimension many planes}
Suppose $\A$ is a generic central hyperplane arrangement in $\RR^3$ of $n$ hyperplanes and moreover we have:
 \begin{enumerate}[label=({\arabic*})]
    \item (Constant smoothness distribution) $\br(H) = r$ for every $H\in\mathscr{A}$, and 
    \item (Power of irrelevant ideal) $\langle \boldsymbol{\ell}^{\br} \rangle = \langle x,y,z \rangle^{r+1}$, and 
    thus $n \geq \binom{r+3}{2}$.
\end{enumerate}
If $\Sigma^\A$ is the fan of $\A$, 
then for all $d\ge 0$ the dimension of the spline space $\calS^{\br}_d(\Sigma^\A)$ is 
\begin{equation*}
\dim\calS^{r}_d(\Sigma^\A)
=
\binom{d+2}{2}
+
2
\binom{n}{2}\binom{d-2r}{2}
+
\begin{cases}
0 & \text{if } d\le r \\[6pt]
n \binom{d + 1 - r}{2}
  & \text{if }r+1\le d\le 2r+1\\[6pt]
\binom{r+3}{2}\binom{d+1-r}{2}-\dbinom{\binom{r+3}{2}}{2} 
& \text{if } d=2r+2\\[6pt]
\binom{d+2}{2} & \text{if } d\ge 2r+3.
\end{cases}
\end{equation*}
\end{theorem} 

\begin{proof}
The formula follows from \cref{thm.SplineGenericArrangement} and \cref{corollary: nonminimal first koszul}.
\end{proof}

\section{$C^0$ and $C^1$ trivariate splines on the fan of a generic arrangement}\label{sec:C0C1}

In this section we record dimension formulas for $\dim \calS^0_d(\Sigma^\A)$ and $\dim \calS^1_d(\Sigma^\A)$ when $\A$ is a generic hyperplane arrangement.  These two cases are particularly interesting as they are used in applications such as the finite element method \cite{Strang,Homology}.

\begin{theorem}\label{thm:C0}
Let $\A$ be a generic hyperplane arrangement 
in $\RR^3$ with $n\ge 3$ hyperplanes.  Then $\dim \calS^0_0(\Sigma^\A)=1, \dim \calS^0_1(\Sigma^\A)=n+3,$ and 
\[
\dim \calS^0_d(\Sigma^\A)=2\binom{d+2}{2}+2\binom{n}{2}\binom{d}{2} \quad \text{ for } d\ge 2.
\]
\end{theorem}
\begin{proof}
Since $\A$ has at least three hyperplanes and is generic, the ideal $\langle \ell_H:H\in\A\rangle$ is the maximal ideal of 
$\RR[x,y,z]$.  Thus we may apply \Cref{theorem: spline dimension many planes} with $r=0$.  Simplifying yields the statement of the theorem.
\end{proof}

If $\A$ is a generic hyperplane arrangement with $n$ hyperplanes, then $\dim \calS^1_d(\Sigma^\A)$ depends only on $n$ and $d$ if $n\le 6$ or $d\neq 4$.  If $n\ge 7$, then  there are two cases for $\dim \calS^1_4(\Sigma^\A)$ depending on the geometry of the hyperplanes.  We record the details in the following theorem.

\begin{theorem}\label{thm:C1}
Let $\A$ be a generic hyperplane arrangement with $n\ge 3$ hyperplanes and $\mathcal{X}=\{P_1,\ldots,P_n\}\subset \mathbb{P}^2(\RR)$ the set of points dual to the linear forms defining $\A$.  Then
\[
\dim \calS^1_d(\Sigma^\A)
=
2\binom{d+2}{2}
+
2\binom{n}{2}\binom{d-2}{2}
+
\eta_d(\A),
\]
where \(\eta_d(\A)=0\) for \(d\ge 6\),
\[
\eta_0(\A)=-1,\qquad
\eta_1(\A)=-3,\qquad
\eta_2(\A)=n-6,\qquad
\eta_3(\A)=3n-10,
\]
\[
\eta_4(\A)=A_n+\delta_n(\A),\qquad
\eta_5(\A)=\epsilon_n,
\]
with \(A_3=0\), \(A_4=3\), \(A_5=5\), and \(A_n=6\) for \(n\ge6\); 
\(\epsilon_4=1\) and \(\epsilon_n=0\) for \(n\neq4\); and
\(\delta_n(\A)=n-6\) if \(n\ge6\) and \(\mathcal X\) lies on a conic, and \(0\) otherwise.
\end{theorem}
\begin{proof}
For $n=3$ and $4$, the theorem follows from a straightforward application of \Cref{thm:threefourHyperplanes}.  For $n=5$, we proceed via \Cref{thm:fiveHyperplanes}.  We have
\begin{center}
    \renewcommand{\arraystretch}{1.5}
    \begin{tabular}{c| c c c c c c c c c}
       $d$  &  0 &1&2&3&4&5&6&7&$\ge 8$\\
       \hline
       
$\sum_{A\subseteq\A} (-1)^{|A|}\binom{d+2-|A|-\br(A)}{2}$&1&3&1&$-5$&$-5$&1&3&1&$0$ 
    \end{tabular}.
    \end{center}
The final remaining term in \Cref{thm:fiveHyperplanes} is $\dim (R/\langle \bm{\ell}^\br\rangle)_{\br(\A)+|\A|-3-d}=\dim (R/\langle \bm{\ell}^1\rangle)_{7-d}$.  In this case, 
$\langle \bm{\ell}^1\rangle =\langle \ell_1^2,\ell_2^2,\ell_3^2,\ell_4^2,\ell_5^2\rangle$, where $\ell_1,\ldots,\ell_5$ 
are the five linear forms defining the hyperplanes of $\A$.  We claim that $\{\ell_i^2\}_{i=1}^5$ must be linearly independent.  By \Cref{thm:Apolarity}, $\dim(R/\langle \bm{\ell}^1\rangle)_2=\dim(\mathfrak{p}_1\cap\cdots\cap \mathfrak{p}_5)_2$, where $\mathfrak{p}_i$ is the ideal of the dual point to $\ell_i$ (in $\mathbb{P}^2(\RR)$).  The condition that no three hyperplanes of $\A$ intersect in a line translates to the fact that no three of the dual points lie on a line.  Suppose $\{\ell_i^2\}_{i=1}^5$ are linearly dependent.  Then $\dim(\mathfrak{p}_1\cap\cdots\cap \mathfrak{p}_5)_2=\dim(R/\langle\bm{\ell}^1\rangle)_2\ge 2$ so there are two linearly independent quadrics $Q_1,Q_2$ that vanish on the five dual points.  
If either $Q_1$ or $Q_2$ is reducible, then it must be a product of two linear factors.  At least one of these linear factors must vanish on three of the dual points, contradicting the fact that $\A$ is generic.  So $Q_1$ and $Q_2$ are both irreducible.  By Bezout's theorem, they intersect in four points, contradicting the fact that we have an ideal of five points.  Hence $\{\ell_i^2\}_{i=1}^5$ are linearly independent. 
The graded pieces of $R/\langle \ell_i^2: 1\le i\le 5\rangle$ 
vanish outside degrees $0,1,2$, with values $1,3,1$ at 
$d=0,1,2$ respectively; consequently 
$\dim(R/\langle \ell_i^2: 1\le i\le 5\rangle)_{7-d}$ 
vanishes outside $d\in\{5,6,7\}$, and equals $1,3,1$ 
for $d=5,6,7$ respectively.
Thus, by \Cref{thm:fiveHyperplanes},
\begin{equation}\label{eq:5generic}
\dim \calS^1_d(\Sigma^\A)=2\tbinom{d+2}{2}+20\tbinom{d-2}{2}+\rho_d,
\end{equation}
where $\rho_0=-1$, $\rho_1=-3$, $\rho_2=-1$, $\rho_3=\rho_4=5$, 
and $\rho_d=0$ for $d\ge 5$.

Now suppose $n\ge 6$ and $\mathcal{X}$ does lie on a conic.  Put $I=\langle \bm{\ell}^{\bm 1}\rangle=\langle \ell_i^2: 1\le i\le n\rangle$.  Since $\mathcal{X}$ lies on a conic, we may assume without loss of generality that $I=\langle \ell_i^2: 1\le i\le 5\rangle$.  By the same argument that we gave in the $n=5$ case, $\{\ell^2_i\}_{i=1}^5$ must be a linearly independent set.  Let $\bm{\ell}'=\{\ell_i\}_{i=1}^5$. By \Cref{prop: first Koszul homology nonminimal},
$
\HK_1(\bm{\ell}^{\bm 1})\cong \HK_1((\bm{\ell}')^{\bm 1})\oplus \bigoplus\limits_{i=6}^n \dfrac{R}{I}(-2).
$
Furthermore $\HK_0(\bm{\ell}^{\bm 1})=R/I=\HK_0((\bm{\ell}')^{\bm 1})$. Thus
\[
\dim \HK_1(\bm{\ell}^{\bm 1})_d-\dim \HK_0(\bm{\ell}^{\bm 1})_d=(n-5)\dim\Big(\frac{R}{I}\Big)_{d-2}+\dim\HK_1((\bm{\ell}')^{\bm 1})_d-\dim\HK_0((\bm{\ell}')^{\bm 1})_d.
\]
The expression $\rho_d$ in \Cref{eq:5generic} is 
the formula for $\dim\HK_1((\bm{\ell}')^{\bm 1})_d-\dim\HK_0
((\bm{\ell}')^{\bm 1})_d$. 
Using the Hilbert function of $R/I$ computed above (in the $n=5$ case), a direct calculation gives 
$(n-5)\dim(R/I)_{d-2}=0$ for $d\le 1$ or $d\ge 5$, and equals 
$n-5$, $3(n-5)$, $n-5$ for $d=2,3,4$ respectively. 
Adding to \Cref{eq:5generic} yields
\[
\dim \calS^1_d(\Sigma^\A)=2\tbinom{d+2}{2}+2\tbinom{n}{2}
\tbinom{d-2}{2}+\eta_d(\A),
\]
where $\eta_0(\A)=-1$, $\eta_1(\A)=-3$, $\eta_2(\A)=n-6$, 
$\eta_3(\A)=3n-10$, $\eta_4(\A)=n$, and $\eta_d(\A)=0$ 
for $d\ge 5$.

If $n\ge 6$ and $\mathcal{X}$ does not lie on a conic, then, as a consequence of \cref{thm:Apolarity}, $\langle \bm{\ell}^\br\rangle=\langle x,y,z\rangle^2$.  The theorem then follows from \Cref{theorem: spline dimension many planes}.
\end{proof}

\section*{Acknowledgments}  This project began
at the ``Workshop on the Applications of Commutative Algebra'' held at the Fields Institute
in Toronto, Canada in May 2025.  All the authors would like
to thank the organizers, Elisa Gorla, Mateusz Micha\l{}ek, and  Hal Schenck,  and 
the Fields Institute for both its hospitality and financial support. 
We also thank Laura Casabella, Anna Natalie Chlopecki, and Hasan Mahmood who participated with our group at Fields, and we thank Claudiu Raicu for suggesting the reference \cite{Bruns-Conca-Romer-2011} that was used in \cref{p:Koszul-homology-powers}. We also thank Peter Alfeld for comments on an early draft.

Checa's research is funded by the European Union under the Grant Agreement number 101044561, POSALG \footnote{Views and opinions expressed do not necessarily reflect those of the European Union or European Research Council (ERC)}.  DiPasquale is partially supported by NSF grant DMS–2344588, and
	he thanks the Fields Institute for funding from the  ``Fields Opportunities for Collaboration - US (FOCUS)" to collaborate with Van Tuyl at McMaster University in early 2026. Nguy$\tilde{\text{\^e}}$n acknowledges support from the AMS-Simons Travel Grant. Van Tuyl’s research is supported by NSERC Discovery Grant 2024-05299.

\appendix

\section{Examples}
\label{sec:Examples}

\begin{example}\label{ex:A3HP}
The braid arrangement $A_3$ is a three-dimensional reflection arrangement.  It consists of the six hyperplanes $H_1=H_x,H_2=H_y,H_3=H_z,H_4=H_{x-y},H_5=H_{x-z},$ and $H_6=H_{y-z}$.  The $A_3$ arrangement is \textit{not} generic since there are four lines along which three hyperplanes intersect: for instance, $H_1,H_2,$ and $H_4$ intersect along the line $x=y=0$.  The fan $\Sigma^{A_3}$ has $24$ three-dimensional cones, $36$ two-dimensional cones, and $14$ rays.  Each of the hyperplanes $H_1,\ldots,H_6$ is a union of six two-dimensional cones of $\Sigma^{A_3}_2$.  Each line in the intersection lattice of $A_3$ is a union of two rays of $\Sigma^{A_3}_1$.

We consider the smoothness distribution $\br(H_1)=\br(H_2)=\br(H_5)=\br(H_6)=6$ and $\br(H_3)=\br(H_4)=5$.  For this distribution, 
\begin{equation}\label{eq:facecontributions}
\sum_{\tau\in\Sigma^{A_3}_2} \dim J(\tau)=12\binom{d+2-6}{2}+24\binom{d+2-7}{2}=12\binom{d-4}{2}+24\binom{d-5}{2}.
\end{equation}

The ideals $J(\tau)$, $\tau\in\Sigma^{A_3}_{1}$, come in a few different types.  There are three lines in $\calL_1(A_3)$ which are the intersection of two hyperplanes: $H_1\cap H_6,H_2\cap H_5,$ and $H_3\cap H_4$.  Each of these contribute two rays to $\Sigma^{A_3}_1$.  The ideals of the two rays in $H_1\cap H_6$ have the Hilbert function
\begin{equation}\label{eq:Koszul1}
\dim \langle x^7,(y-z)^7\rangle_d=2\binom{d-5}{2}-\binom{d-12}{2}.
\end{equation}

The ideals of the two rays in $H_2\cap H_5$ also have the Hilbert function in \cref{eq:Koszul1}, while the ideals of the two rays in $H_3\cap H_4$ have the Hilbert function
\begin{equation}\label{eq:Koszul2}
\dim \langle z^6,(x-y)^6\rangle_d=2\binom{d-4}{2}-\binom{d-10}{2}.
\end{equation}
There are four other lines in the intersection lattice $\calL_1(A_3)$; each of these is the intersection of three planes.  For example, consider the intersection $H_1\cap H_2\cap H_4$.  The ideal of both rays in this intersection has the form
$
\langle x^7,y^7,(x-y)^6\rangle.
$
According to \cite[Theorem~2.7]{FatPoints}, this ideal has a minimal free resolution of the form
\[
R(-7)^2\oplus R(-6)\leftarrow R(-10)^2\leftarrow 0.
\]
Thus
\begin{equation}\label{eq:tripleints}
\dim \langle x^7,y^7,(x-y)^6\rangle_d=\binom{d-4}{2}+2\binom{d-5}{2}-2\binom{d-8}{10}.
\end{equation}
Putting together Equations \cref{eq:facecontributions}, \cref{eq:Koszul1}, \cref{eq:facecontributions}, \cref{eq:Koszul2}, and \cref{eq:tripleints}, we have
\begin{align*}
 \sum_{\tau\in\Sigma^{A_3}_2}\dim J(\tau)_d&-\sum_{\gamma\in\Sigma^{A_3}_1}\dim J(\gamma)_d= 12\binom{d-4}{2}+24\binom{d-5}{2}
 -4\left(2\binom{d-5}{2}-\binom{d-12}{2}\right) \\
 &-2\left(2\binom{d-4}{2}-\binom{d-10}{2}\right)-8\left(\binom{d-4}{2}+2\binom{d-5}{2}-2\binom{d-8}{2}\right)\\
 =& 16\binom{d-8}{2}+2\binom{d-10}{2}+4\binom{d-12}{2}.
\end{align*}
By \Cref{thm.SplineGenericArrangement},
\[
\dim \calS^\br(\Sigma^{A_3})=2\binom{d+2}{2}+16\binom{d-8}{2}+2\binom{d-10}{2}+4\binom{d-12}{2}=12d^2-204d+1000
\]

for $d\gg 0$.  A computation in Macaulay2~\cite{M2} shows that $\calS^\br(\Sigma^{A_3})$ is actually free with

\[
\calS^{\br}(\Sigma^{A_3})\cong R\oplus R(-6)^2\oplus R(-7)^4\oplus R(-10)^8\oplus R(-12)\oplus R(-13)^4\oplus R(-14)^4.
\]
\end{example}

\begin{example}\label{ex:A3reg}
Consider the $A_3$ arrangement with smoothness distribution as in \Cref{ex:A3HP}.  For ease of reference, $H_1=H_x,H_2=H_y,H_3=H_z,H_4=H_{x-y},H_5=H_{x-z},$ and $H_6=H_{y-z}$ with $\br(H_i)=6$ for $i=1,2,5,6$ and $\br(H_i)=5$ for $i=3,4$.  Observe that the hyperplanes $H_1,H_3,H_4,H_6$ defined by linear forms $x,z,x-y,y-z$ are linearly independent and every subset of size three of these is also linearly independent.  
Applying part (2) of \Cref{maintheoremUsingRegularity} and the result of \Cref{ex:A3HP}, we deduce that
$
\dim \calS^\br_d(\Sigma^{A_3})=12d^2-204d+1000
$
for $d\ge 25$.  In \Cref{ex:A3HP} we observed that Macaulay2 calculations imply that $\calS^\br(\Sigma^{A_3})$ is a free $R$-module and the maximum degree of a generator is $14$.  Therefore the above formula actually holds for $d\ge 12$.
\end{example}

For the next example, we recall the algorithm of Catalisano from~\cite{Catalisano-1991}.

\begin{theorem}[{\cite[Theorem~3.1]{Catalisano-1991}}]
\label{thm:catalisanoAlgorithm}
Assume that the points $P_1,\dots,P_n$ for $n \geq 2$ lie on a nonsingular conic and let $m_1 \geq \dots \geq m_n \geq 0$. Consider the ideals
$J = \bigcap_{i = 1}^n \mathfrak{p}_i^{m_i}$
and $J' \coloneqq \bigcap_{i = 1}^n \mathfrak{p}_i^{m_i'}$ where the $m'_i$ are defined as:
$$m_i' = \begin{cases}
    m_i & \sum_{i = 1}^n m_i \leq 2m_1 + 2m_2 - 1, \; i \geq 3\\
    \max\{m_i-1,0\} & \text{ otherwise.}
\end{cases}.$$
Let $t =\max\{m_1 + m_2 - 1, \ceil{\frac{\sum_{i = 1}^nm_i}{2}}\}$ and $\delta = \sum_{i = 1}^n\frac{m_i(m_i - 1)}{2}$. Then,
$$\dim (S/J)_d = \begin{cases}
    \delta & d \geq t \\
    2n + 1 + \dim (S/J')_{d-2} & 0 \leq d < t, \: \sum_{i = 1}^n m_i \geq 2m_1 + 2m_2 \\
    n + 1 + \dim (S/J')_{d-1} & 0 \leq d < t, \: \sum_{i = 1}^n m_i < 2m_1 + 2m_2.
\end{cases}$$
\end{theorem}

\begin{example}\label{ex:5forms2smoothness}
Consider the case of a generic hyperplane arrangement $\A$ with five hyperplanes and the smoothness distribution given by $\br(H) =2$ for each $H$. The ideal given by the sequence of hyperplanes equals $I =\langle \ell_1^3, \ell_2^3, \ell_3^3, \ell_4^3, \ell_5^3 \rangle \subset R$ and let $P_i$ be the points dual to $\ell_i$ for $i = 1,\dots,5$, as in \cref{eq:dualPoint}. 
Denote by $J = \bigcap_{i = 1}^5\mathfrak{p}_i$ where $\mathfrak{p}_i$ is the vanishing ideal of $P_i$ for $i = 1,\dots,5$. By apolarity, we have:
    \[
    \dim \left(R/I \right)_d = \dim \left( J^{(d-2)} \right)_d ,  ~~\mbox{for all $d \geq 0$},
    \]
    where for any $s$, the ideal $J^{(s)} = \mathfrak{p}_1^{s} \cap \mathfrak{p}_2^{s} \cap \mathfrak{p}_3^{s} \cap \mathfrak{p}_4^{s} \cap \mathfrak{p}_5^{s}$ is the ideal consists of all polynomials that vanish at $P_1,P_2,P_3,P_4,P_5$ to order at least $s$. Since the $5$ points lie on an irreducible conic, we can compute explicitly $\dim \left( J^{(d-2)} \right)_d$ using \cref{thm:catalisanoAlgorithm}. We first claim that $\dim \left( J^{(d-2)} \right)_d = 0$, or equivalently, $\dim (S/J^{(d-2)})_d = {d+2 \choose 2}$.

    We can check that the numerical conditions given in \cref{thm:catalisanoAlgorithm} are satisfied. Therefore, $\dim (S/J)_d = 2d+1 + \dim (S/J')_{d-2}$. 
    Using \cite[Theorem 3.1(a)]{Catalisano-1991} again, we have
    \[
    \dim(S/J^{(d)})_d = 2d+1 + \dim(S/J^{(d-2)})_{d-2} = 2d+1 + 2(d-2)+1 + \dim(S/J^{(d-4)})_{d-4},
    \]
Since five points lie on an irreducible conic, $\dim(S/J^{(d-4)})_{d-4} = {d-2 \choose 2}$. It follows that 
    \[
    \dim(S/J^{(d)})_{d}= 2d+1+2(d-2)+1 + {d-2 \choose 2} = {d+2 \choose 2},
    \]
    as desired. Thus, $\dim \left(R/I \right)_d = 0$
    for all $d\ge 5$. 
    Furthermore, since the $5$ points lie on an irreducible conic, $\dim \left(R/I \right)_4 = \dim \left(J \right)_4=1$, hence, $\dim \left(R/I \right)_3=10-5=5$. Lastly, $\dim\left(R/J \right)_2 = \dim S_2=6$, and $\dim\left(R/J \right)_1 = \dim S_1=3$. Therefore, we have     $$\hilb_{\HK_0(\bm{\ell}^{\br})}(t)=1+3t+6t^2+5t^3+t^4.$$ 
    By \Cref{prop:someKoszul}(1), $ \hilb_{\HK_2(\bm{\ell}^{\br})}(t)=t^8+5t^9+6t^{10}+3t^{11}+t^{12}$.
    By \Cref{prop:someKoszul}(2), 
    \[
    \hilb_{\HK_1(\bm{\ell}^{\br})}(t)=\hilb_{\HK_0(\bm{\ell}^{\br})}(t)+\hilb_{\HK_2(\bm{\ell}^{\br})}(t)-\frac{(1-t^3)^5}{(1-t)^3} = t^4+9t^5+12t^6+9t^7+t^8.
    \]
    Thus, we have 
    \begin{center}
    \renewcommand{\arraystretch}{1.5}
    \begin{tabular}{c| c c c c c c c c c c}
       $d$  &  0 &1&2&3&4&5&6&7&8 & $\ge 9$\\
       \hline
       $\dim \HK_1(\ell^2)_d-\dim \HK_0(\ell^2)_d$ &$-1$&$-3$&$-6$&$-5$& 0& 9&12&9&1 &0
    \end{tabular}.
    \end{center}
    
    By \Cref{thm.SplineGenericArrangement}, we obtain
    \begin{center}
    \renewcommand{\arraystretch}{1.5}
    \begin{tabular}{c| c c c c c c c c c c}
       $d$  &  0 &1&2&3&4&5&6&7&8&$\ge 9$\\
       \hline
       $\dim\calS^{2}_d(\Sigma^\A)$ &1&3&6&15&30&51&88&141&211&$11d^2-87d+202$ 
    \end{tabular}.
    \end{center}
\end{example}

\newpage

\end{document}